\documentclass[reqno]{amsart}

\usepackage[utf8]{inputenc}
\usepackage{amsfonts}
\usepackage{amssymb}

\usepackage{amsthm}
\usepackage{mathtools}
\usepackage{color}
\usepackage{verbatim}
\usepackage{cases}

\newcommand{\Div}{{\rm div}\hspace{.5mm}}

\newcommand{\x}{\mathbf{x}}
\newcommand{\y}{\mathbf{y}}
\newcommand{\Y}{\mathbf{Y}}
\newcommand{\z}{\mathbf{z}}
\newcommand{\uu}{\mathbf{u}}
\newcommand{\vv}{\mathbf{v}}

\newcommand{\J}{\mathbf{J}}
\newcommand{\E}{\mathbf{E}}

\newcommand{\HH}{\mathbf{H}}
\newcommand{\F}{\mathbf{F}}
\renewcommand{\l}{\lambda}
\renewcommand{\b}{\beta}

\theoremstyle{plain}
\newtheorem{theorem}{Theorem}[section]
\newtheorem{corollary}{Corollary}[section]

\newtheorem{lemma}{Lemma}[section]
\newtheorem{proposition}{Proposition}[section]
\theoremstyle{definition}

\theoremstyle{remark}

\newtheorem{remark}{Remark}[section]
\numberwithin{equation}{section}

\author[H. Frid]{Hermano Frid}
 \address{Instituto de Matem\'atica Pura e Aplicada - IMPA\\ Estrada Dona Castorina, 110\\
Rio de Janeiro, RJ, 22460-320, Brazil}
\email{hermano@impa.br}
\thanks{H.~Frid gratefully acknowledges the support from CNPq, through grant proc.~303950/2009-9, and FAPERJ, through grant E-26/103.019/2011.}

\author[D.~Marroquin]{Daniel R. Marroquin}
 \address{Instituto de Matem\'atica Pura e Aplicada - IMPA\\ Estrada Dona Castorina, 110\\
Rio de Janeiro, RJ, 22460-320, Brazil}
\thanks{D.~Marroquin thankfully acknowledges the support from CNPq, through grant proc. 150118/2018-0.}
\email{danielrm@impa.br}

\author[J.F.C.~Nariyoshi]{Jo\~ao F.C.~Nariyoshi}
 \address{Instituto de Matem\'atica Pura e Aplicada - IMPA\\ Estrada Dona Castorina, 110\\
Rio de Janeiro, RJ, 22460-320, Brazil}
\thanks{J. F. C.~Nariyoshi appreciatively acknowledges the support from CNPq, through grant proc. 140600/2017-5.}
\email{jfcn@impa.br}

\title[Smooth large solutions for aurora type phenomena in the 2-torus]{Global smooth solutions with large data \\ for a system modeling \\ aurora type phenomena  in the 2-torus}



\begin{document}

\maketitle

\begin{abstract}
We prove existence and uniqueness of smooth solutions with large initial data for a system of equations modeling the interaction of short waves, governed by a nonlinear Schr\"{o}dinger equation, and long waves, described by the equations of magnetohydrodynamics. In the model, the short waves propagate along the streamlines of the fluid flow. This is translated in the system by setting up the nonlinear Schr\"{o}dinger equation in the Lagrangian coordinates of the fluid. Besides,  the equations are coupled by nonlinear terms accounting for the strong interaction of the dynamics.   The system provides a simplified mathematical model for studying aurora type phenomena.  We focus on the 2-dimensional case with periodic boundary conditions.
This is the first result on existence of smooth solutions with large data for the multidimensional case of the model under consideration.
\end{abstract}

\section{Introduction}

We consider the following system of equations modeling the interaction of short waves, governed by a nonlinear Schr\"{o}dinger equation (NLS), and long waves, described by the equations of magnetohydrodynamics (MHD):
\begin{align}
& \rho_t+\text{div}(\rho\uu)= 0,\label{E2rho}\\
&(\rho \uu)_t + \text{div}(\rho\uu\otimes\uu) + \nabla P = \HH\cdot\nabla \HH - \tfrac{1}{2}\nabla|\HH|^2\nonumber\\
&\qquad\qquad\qquad\qquad\qquad\qquad+ \alpha \nabla(g'(1/\rho)h(|\psi\circ \Y|^2) \J_\y/\rho) + \Div \mathbb{S}, \label{E2u}\\
& \HH_t + \uu \cdot \nabla \HH - \HH\cdot \nabla\uu + \HH\Div \uu = \nu \Delta \HH,\label{E2H}\\
&\Div \HH = 0,\label{E2H'}\\
& i\psi_t+\Delta_\y \psi = |\psi|^2\psi + \alpha g(v)h'(|\psi|^2)\psi. \label{E2psi}
\end{align}
Our main goal is to prove existence and uniqueness of solutions with prescribed large initial data in the 2-dimensional torus.

Let us recall that the MHD equations model the dynamics of a conductive fluid in the presence of a magnetic field. Accordingly, in the system above, $\rho=\rho(\x,t) \geq 0$ and $\uu=\uu(\x,t)\in \mathbb{R}^2$ are the fluid's density and velocity, respectively, and $\HH=\HH(\x,t)\in\mathbb{R}^2$ is the magnetic field; $P=P(\rho)$ denotes the pressure and $\mathbb{S}$ is the viscous stress tensor given by
\begin{equation}
\mathbb{S}=\lambda (\Div \uu)\text{Id}+ \mu (\nabla \uu + (\nabla \uu)^\top),\label{defS}
\end{equation}
where $\lambda$ and $\mu$ are the viscosity coefficients which are, in general, functions of the density and must satisfy the relations $\mu >0$ and $3\lambda + 2\mu \geq 0$; $\nu$ is the magnetic diffusivity.

On the other hand, $\psi=\psi(\y,t)\in\mathbb{C}$ is the wave function, $\alpha>0$ is the interaction coefficient and $g$ and $h$ are smooth functions taking nonnegative values.

The MHD equations above are stated in the Eulerean coordinates $(\x,t)$, that describe the dynamics from an outsider's point of view, while the NLS is stated in the Lagrangian coordinates $(\y,t)$ associated to the velocity field of the fluid, which follow the particle paths. 

The Lagrangian transformation $\Y(\x,t)=(\y(\x,t),t)$ can thus be defined through the relation
\begin{equation}
\y(\Phi(\x,t),t)=\y_0(\x),\label{deflagr}
\end{equation}
where $\y_0$ is any diffeomorphism, which can be chosen conveniently, and $\Phi:\Omega\to \Omega$ is the flux associated to the fluid's velocity field, given by
\begin{equation}
\begin{cases}
\frac{d\Phi}{dt}(\x,t)=u(\Phi(\x,t),t),\\
\Phi(\x,0)=\x.
\end{cases}\label{flowofu}
\end{equation}

Finally, $\mathbf{J}_\y$ is the Jacobian determinant $\det (\frac{\partial\y}{\partial\x}(\x,t))$ and $v(\y,t)$ is the specific volume determined by the identity
\[
v(\y(\x,t),t)=\frac{1}{\rho(\x,t)}.
\]

The interaction terms in equations \eqref{E2u} and \eqref{E2psi} are an external force term and a potential due to external forces, respectively, that account for the strong interaction between the dynamics.

This kind of model of short wave-long wave interactions in compressible fluid dynamics was first introduced by Dias and Frid in \cite{DFr}, where, motivated by Benney's general theory on short wave-long wave interactions \cite{Be}, they proposed a similar coupling where the long waves were given by the Navier-Stokes equations for a compressible isentropic fluid and proved existence and uniqueness of solutions in the 1-dimensional setting. They also studied the vanishing viscosity problem. An extension of Benney's theory for systems of conservation laws was given in \cite{DFF}. 

Later, Frid, Pan and Zhang \cite{FrPZ} studied the full 3-dimensional equations proving global existence and uniqueness of smooth solutions with small data. After this, Frid, Jia and Pan, extended these results to the model involving the MHD equations, proving existence, uniqueness and decay rates of solutions with small data, also in the 3-dimensional case.

More recently, Frid, Marroquin and Pan \cite{FrDRMP} proposed an approximation scheme consisting of a regularization of the system above, proving existence of weak solutions and showing the convergence of the sequence to a solution of the limit decoupled system as the regularizing parameters tend to zero together with the interaction coefficient, all  in a bounded domain of $\mathbb{R}^2$.

Finally, in \cite{DRM} Marroquin extended the results in \cite{DFr} to the case where the long waves are given by a compressible, heat conductive magnetohydrodynamic fluid, also proving existence and uniqueness of solutions and studying the vanishing viscosity problem, in the planar (1-dimensional) case.

The system above models an aurora type phenomenon where a small wave, obeying a nonlinear Schr\"{o}dinger equation, propagates along the streamlines of a magnetohydrodynamic fluid. To this end, the NLS equation is stated in the Lagrangian coordinates of the fluid, which by definition, follow the trajectories of the particles. However, this poses several limitations, as the Lagrangian transformation becomes singular in the presence of vacuum. Indeed, it can be shown that the Jacobian determinant of the change of variables $\mathbf{J}_\y$ vanishes if and only if the density vanishes. As in the multidimensional setting solutions of the fluid equations may not be regular enough and vacuum may appear in finite time, the study of the model under consideration is challenging.

In this paper, building on the pioneering work by Va\u \i gant and  Kazhikhov  \cite{VK} and its posterior improvements, e.g.,  \cite{P,JWX,HL,Mei, Mei'}, as well as on the small gem classical paper by Brezis and Gallouet \cite{BrGa}, we prove the existence and uniqueness of strong solutions to the full system \eqref{E2rho}-\eqref{E2psi} without restrictions on the size of the initial data, in the 2-dimensional case with periodic conditions. Under certain hypotheses on the viscosity coefficients, we prove, in particular, that there is no vacuum, nor concentration, if this is the case initially.

The most important feature of this model is that it is endowed with an energy identity, which can be stated in differential form as
\begin{flalign}
 &\Big\{ \big( \rho( \frac{1}{2}|\mathbf{u}|^2 + e )+\frac{1}{2}|\mathbf{H}|^2 \big)_t + \big(\mu |\nabla_\mathbf{x} \mathbf{u}|^2+(\lambda(\rho)+\mu)(\text{div}_\mathbf{x}\mathbf{u})^2 + \nu |\nabla \mathbf{H}|^2\big) \label{difE}\\
 &\hspace{15mm} + \text{div}_{\mathbf{x}}\big( \mathbf{u}(\rho(\frac{1}{2}|\mathbf{u}|^2 + e) + p+\alpha g'(1/\rho)h(|\psi\circ \mathbf{Y}|^2)\frac{\J_\y}{\rho} \big)\nonumber \\
 &\hspace{30mm}- \text{div}_{\mathbf{x}}\big(\mathbb{S}\cdot \mathbf{u} +  (\HH\otimes \HH)\cdot \uu -\uu|\HH|^2 + \nu \nabla\HH\cdot\HH )\big)\Big\} d\mathbf{x}\nonumber \\
 &\hspace{6mm}= -\Big\{ \big( \frac{1}{2}|\nabla_{\mathbf{y}}\psi(t,\mathbf{y})|^2 + \frac{1}{4}|\psi(t,\mathbf{y})|^4 + \alpha g(v(t,\mathbf{y})) h(|\psi(t,\mathbf{y})|^2 \big)_t \nonumber\\
 &\hspace{65mm}- \text{div}_{\mathbf{y}}(\overline{\psi}_t\nabla_{\mathbf{y}}\psi + \psi_t\nabla_{\mathbf{y}}\overline{\psi})\Big\}d\mathbf{y},&& \nonumber
\end{flalign}
where, $e=e(\rho)$ is the internal energy given by
\[
e(\rho):=\int^\rho \frac{p(s)}{s^2}ds.
\]

Indeed this identity is obtained by multiplying the momentum equation \eqref{E2u} by $\uu$, the NLS \eqref{E2psi} by $\overline{\psi}_t$ (the complex conjugate of $\psi_t$) taking real part and adding the resulting equations. We omit the details as they follow the same lines as in any of the references on the model cited above (\cite{DFr,FrPZ,FrJP,FrDRMP,DRM}).

In order to state precisely our results, let us fix some notation. We consider our spatial domain to be the unitary square in $\mathbb{R}^2$ endowed with periodic boundary conditions
\begin{align*}
\Omega&=\{\x = (x_1,x_2) \in\mathbb{R}^2 ; 0\le x_1\le 1, 0\le x_2\le 1\},
\end{align*}
with opposite sides identified, each point with its antipode, so that one could understand $\Omega$ as being the two-dimensional torus $\Omega = \mathbb{T}^2$.

Choosing $\y_0(\x)=\x$, $\x\in\Omega$, in the definition of the Lagrangian coordinate \eqref{deflagr}, we see that $\y(\cdot,t)$ is a change of variables from $\Omega$ to itself (provided that $\rho$ is strictly positive and finite). However, as we deal with Sobolev norms in the different coordinates, we have to distinguish between one from the other and thus we denote the domain of the Lagrangian coordinate as $\Omega_\y$. 

As aforementioned, the viscosity coefficients are, in general, functions of the density, and in this work we assume that 
\begin{align}
\mu(\rho) &= \text{ const.} = \mu > 0, \nonumber \\
\lambda(\rho) &= b \rho^\beta \text{, where } \beta > 4/3 \text{ and } b > 0, \label{assumptionvisc}
\end{align}
Regarding the pressure, we assume the following constitutive relation
\begin{equation}
p(\rho)=a\rho^\gamma,\label{assumptionp}
\end{equation}
where $a>0$ and $\gamma>1$.

These assumptions agree with those in  \cite{VK} and in   \cite{P} on the 2-dimensional periodic Navier-Stokes equations; and also with  those in \cite{Mei,Mei'} on the 2-dimensional periodic MHD equations.

From a mathematical standpoint, the condition on $\mu$ ensures that the associated Lamé operator $L\uu = \Div \mathbb{S}$ is uniformly parabolic, whereas the hypothesis on $\lambda$ hinders singularities in the density $\rho$. The latter consequence is crucial to our model, as the non-degeneracy of the Lagrangian coordinates is intimately linked to the regularity of the density.

At last, we assume that the interaction coefficient $\alpha$ is a positive constant and that the coupling functions $g$ and $h$ are smooth and satisfy
\begin{equation}
\begin{cases}
g,h:[0,\infty)\to[0,\infty), &g(0)=h(0)=0, \\
\text{supp}\,g' \text{ compact in }(0,\infty), &\\
\text{supp}\,h' \text{ compact in }[0,\infty)
\end{cases}\label{assumptiongh}
\end{equation}

Under these assumptions we consider the initial value problem for system \eqref{E2rho}-\eqref{E2psi} on the square $\Omega$ with periodic boundary conditions, and subject to initial data
\begin{equation}
(\rho,\uu,\HH)|_{t=0} = (\rho_0,\uu_0,\HH_0)(\x),\qquad \psi_0|_{t=0}=\psi_0(\y). \label{initialdata}
\end{equation}

Our first main result reads as follows.
\begin{theorem} \label{principalthm}
	Let $m \geq 3$ be an integer and assume, in addition to \eqref{assumptionvisc},  \eqref{assumptionp} and \eqref{assumptiongh}, that the periodic data $\rho_0$, $\uu_0$, $\HH_0$ and $\psi_0$ on \eqref{initialdata} satisfy
	\begin{equation}
	\begin{cases}
	\rho_0 &\in H^{m}(\Omega), \\
	\uu_0 &\in H^m(\Omega) , \\
	\HH_0 &\in H^m(\Omega) , \\
	\psi_0 &\in H^m(\Omega_\y). 
	\end{cases}\label{secondinitialdata}
	\end{equation}
	and the initial data has no vacuum: there exists constants $0< m_0 < M_0$ such that
	\begin{equation}
	m_0 \leq \rho_0(\x) \leq M_0
	\end{equation}
	for any $\x \in \Omega$.
	
	Then there exists a unique strong solution to the system \eqref{E2rho}-\eqref{E2psi} $(\rho, \uu,\HH, \psi)$ lying on the space
	\begin{equation}
	\begin{cases}
	\rho &\in C([0,\infty); H^m(\Omega))\cap C^1([0,\infty);H^{m-1}(\Omega)),  \\
	\uu  &\in L_{\text{loc}}^2(0,\infty;H^{m+1}(\Omega))\cap H_{\text{loc}}^1(0,\infty;H^{m-1}(\Omega)),  \\
	\HH &\in L_{\text{loc}}^2(0,\infty;H^{m+1}(\Omega))\cap H_{\text{loc}}^1(0,\infty;H^{m-1}(\Omega)),  \\
	\psi &\in C([0,\infty);H^m(\Omega))\cap C^1( [0,\infty);H^{m-2}(\Omega)).
	\end{cases}\label{regularityofu}
	\end{equation}
	
	And, for any $T>0$ we have the estimates
	\begin{equation}
	M^{-1} \leq \rho(\x,t) \leq M \label{novacuum}
    \end{equation}
	for all $\x \in \Omega$ and $0\leq t\leq T$, and  
		\begin{equation}
	\Vert \rho(t) \Vert_{H^m(\Omega)}, \Vert \uu(t) \Vert_{H^m(\Omega)}, \Vert \HH(t) \Vert_{H^m(\Omega)}, \Vert \psi(t) \Vert_{H^m(\Omega_\y)} \leq C, \label{stability}
	\end{equation}
	for all $0\leq t\leq T$, where $M$ is a positive constant depending only on $T$, $\| \uu_0 \|_{H^1(\Omega)}$, $\| \HH_0\|_{H^1(\Omega)}$, $\| \psi_0 \Vert_{H^2(\Omega_\y)}$, $m_0$ and $M_0$; and $C$ is a positive constant depending on $T$, $\Vert \rho_0 \Vert_{H^m(\Omega)}$, $\Vert \uu_0 \Vert_{H^m(\Omega)}$, $\|\HH\|_{H^m(\Omega)}$, $\Vert \psi \Vert_{H^m(\Omega_\y)}$, $m_0$ and $M_0$.
	
	Furthermore, there is a constant $N$ depending on $T$, $\|\rho_0\|_{H^2(\Omega)}$, $\| \uu_0 \|_{H^2(\Omega)}$, $\| \HH_0\|_{H^2(\Omega)}$, $\| \psi_0 \Vert_{H^2(\Omega_\y)}$, $m_0$ and $M_0$ such that
	\begin{equation}
	\Vert \rho(t) \Vert_{H^2(\Omega)}, \Vert \uu(t) \Vert_{H^2(\Omega)}, \Vert \HH(t) \Vert_{H^2(\Omega)}, \Vert \psi(t) \Vert_{H^2(\Omega_\y)} \leq N. \label{stability2}
	\end{equation}
\end{theorem} 

Observe that \eqref{regularityofu} implies that
$$	\uu,\HH  \in C([0,\infty);H^{m}(\Omega))\cap C^1 ([0,\infty);H^{m-2}(\Omega)),$$
so that the well-posedness statement makes sense. 

It may be noted that it implies by an interpolation argument (see Theorem~\ref{dependcont}), that the solution map 
\[(t, \rho_0, \uu_0,\HH_0, \psi_0) \mapsto (\rho(t), \uu(t),\HH(t), \psi(t))\]
is continuous in the strong topology of $[0,\infty)\times H^m(\Omega)^3 \times H^{m}(\Omega_\y)$ into the strong topolgy $H^{m-1}(\Omega)^3 \times H^{m-1}(\Omega_\y)$, for all integers $m \geq 3$.

In particular, if one identifies $C^\infty(\Omega)$ with the Fréchet space $\cap_{m=0}^\infty H^m(\Omega)$, not only the solution map of \eqref{E2rho}--\eqref{E2psi} is \textit{well-defined} on $[0,\infty)\times C^\infty(\Omega)^3 \times C^\infty(\Omega_\y)$, but also turns out to be  \textit{continuous in its topology}.

The proof of Theorem \ref{principalthm} consists of a local result and the extraction of a priori estimates, which serve the purpose of extending the local solutions globally in time.

For the local result we adapt the framework by V. A. Solonnikov \cite{Solonnikov} on the Navier-Stokes equations. More precisely, we obtain a local solution of the system \eqref{E2rho}--\eqref{E2psi} as a fixed point of a transformation $K$ defined as follows:
\begin{itemize}
	\item Given a velocity field $\uu$, one solves the continuity equation \eqref{E2rho} with initial data $\rho(\x,0)=\rho_0(\x)$ (observe that the continuity equation is a \textit{linear} first order partial differential equation, and thus can be solved \textit{explicitly} by the method of characteristics);
	\item Once that $\rho$ is available, one can solve the equation for $\psi$ \eqref{E2psi} (with $\Y$ being the Lagrangian transformation associated to the given velocity $\uu$) by means of a variant of the method of H. Brézis--Th. Gallouet \cite{BrGa};
	\item One also solves the magnetic field equations \eqref{E2H}, \eqref{E2H'} in terms of $\uu$. Note that the magnetic field equation \eqref{E2H} reduces to a linear parabolic equation and can be solved by standard methods.
	\item Finally, one ``returns'' $K\uu = \vv$ as the solution of the \textit{linear} parabolic problem
	\begin{equation}
	\begin{cases}
	\rho \vv_t + \rho (\uu \cdot \nabla) \vv + \nabla P_0 = \alpha \nabla (g'(1/\rho)h(|\psi \circ \Y|^2) \J_\y/\rho) + \Div \mathbb{S},\\
	\vv(\x,0) = \uu_0(\x)
	\end{cases}\label{E2ulinear}
	\end{equation}
\end{itemize}

Clearly, a fixed point of such application will be a solution to the equation we are investigating. 

Let us point out that, unlike the case  addressed by  Solonikov, here the viscosity is not constant; in fact it depends on the densisty (recall \eqref{assumptionvisc}). This brings new subtleties to the estimates. In addition to this, the presence of the magnetic field, the wave function and the coordinate change between Eulerean and Lagrangian coordinates bring extra difficulties. 

In particular, the $H^m$ regularity, with $m\geq 3$, in Theorem~\ref{principalthm} is necessary for the fixed point argument in the proof of the existence and uniqueness of local solutions, as, in order to ensure that the application $K$ is a contraction, we  require the wave function and the specific volume to be Lipschitz continuous. This is due to the coupling terms, where we have a composition of the wave function with the Lagrangian transformation.

Concerning the extension of the local solution to a global one, as usual, it follows from a priori estimates which allow reiteration of the construction made for the local solution.  The starting point in the deduction of the a priori estimates is the energy identity \eqref{difE}. After this, we prove the crucial uniform estimates from above and away from vacuum for the density. This was the decisive step first achieved in the pioneering paper of Va\u \i gant and Kazhikhov \cite{VK}, where it was  shown that the prescription 
$\l=\rho^\b$, with $\b>3$,  yielded such bounds; later on the  restriction on the power $\b$ was relaxed to $\b> 4/3$ by Huang and Li in \cite{HL}. Here we build on the estimates by Yu Mei \cite{Mei,Mei'} for the $2$-dimensional periodic MHD equations which, in turn, are inspired by \cite{VK} and its improvements by  Perepelitsa in \cite{P}, Jiu, Wang and Xin \cite{JWX} and \cite{HL} on the 2-dimensional periodic Navier-Stokes equations. 

It is worth while to recall the ingenious way in which the bounds from above and below for the density were obtained in \cite{VK}.  In \cite{H}, Hoff  introduced and emphasized the striking regularity properties of the effective viscous flux ${\bf F}$, which for the isentropic Navier-Stokes is defined by ${\bf F}=(2\mu+\l(\rho))\operatorname{div} u-P$.   In  \cite{VK},  the  transport equation 
$$
( \theta(\rho))_t+u\cdot\nabla \theta(\rho)+P(\rho)+{\bf F}=0
$$
is considered, where $\theta(\rho)=\int_1^\rho\frac1s(2\mu+\l(s))\,ds=2\mu\log\rho+\frac1\b(\rho^\b-1)$, which easily follows from the continuity equation. Therefore, the problem of obtaining the referred bounds for the density is reduced, using the nonnegativity of the pressure,  to that of  obtaining a bound for the integral in time of the $L^\infty$ norm of ${\bf F}$. Indeed,  one gets first  a bound from above for $\theta(\rho)$, which implies a bound from above for the density, which in turn, together with the bound for 
${\bf F}$, can be  used to give also a bound from below for $\theta(\rho)$ and so for $\rho$.  Obtaining the bound for ${\bf F}$  is not at all an easy task, but was  achieved in \cite{VK}. In passing, it also becomes  transparent the importance of the introduction of the dependence of $\l$ on $\rho$ in the form $\rho^\b$.

For the MHD equations, in \cite{Mei, Mei'},  the desired uniform bounds on the density follow from  key estimates on the commutators $[u_i,R_i R_j](\rho u_j)$ and $[H_i,R_i R_j](H_j)$ of Riesz transforms and the operators of multiplication by $u_i$ and $H_i$. Such estimates ultimately depend on a careful analysis of the effective viscous flux $\F$, which, in turn, relies on some $L^p$ bounds on the density. For the system  \eqref{E2rho}--\eqref{E2psi} considered here, the definition of effective viscous flux, whose treatment follows closely that of the MHD case,  takes the form \eqref{EVF} which, comparing with that for the isentropic Navier-Stokes system, includes the term $-\frac12|H|^2$, as in the MHD case, and also the term $-\alpha g'(1/\rho)h(|\psi\circ \Y|^2) \J_\y/\rho$, due to the coupling with the NLS equation, whose nature is  also that of a pressure.

We stress that the strong coupling between the MHD equations and the NLS  poses a delicate interplay between the mechanics. On top of the nonlinear coupling terms from the momentum equation and the NLS, the equations are also coupled through the coordinate change. At this point the regularity of the Lagrangian transformation comes into play and is crucial in order to close the estimates.

Once we have estimates \eqref{novacuum} and \eqref{stability2}, we proceed to prove estimate \eqref{stability} by induction, thus obtaining the necessary estimates in order to extend the local solutions to global ones. It should be pointed out, that the lifespan of the local solutions depends only on the strict positiveness and finiteness of the initial density as well as the size of initial data in the corresponding Sobolev norms. Thus estimates \eqref{novacuum} and \eqref{stability} are sufficient to obtain the global solutions from the local ones.

In this connection,  with the estimates \eqref{novacuum} and \eqref{stability2} at hand, we can also prove the existence of global solutions under weaker hypotheses on the initial data.

\begin{theorem}\label{secondthm}
Assume that \eqref{assumptionvisc},  \eqref{assumptionp} and \eqref{assumptiongh} hold and suppose that the periodic initial data $\rho_0$, $\uu_0$, $\HH_0$ and $\psi_0$ on \eqref{initialdata} satisfy 
	\begin{equation}
	\begin{cases}
	\rho_0 &\in H^2(\Omega),\\
	\uu_0 &\in H^2(\Omega) , \\
	\HH_0 &\in H^2(\Omega) , \\
	\psi_0 &\in H^2(\Omega_\y). 
	\end{cases}
	\end{equation}
	and that there are constants $0< m_0 < M_0$ such that
	\begin{equation}
	m_0 \leq \rho_0(\x) \leq M_0
	\end{equation}
	for any $\x \in \Omega$.
	
	Then, there exists a solution $(\rho, \uu,\HH, \psi)$ to system \eqref{E2rho}-\eqref{E2psi} satisfying \eqref{regularityofu}, \eqref{novacuum} and \eqref{stability} with $m=2$.
\end{theorem}

The proof of Theorem \ref{secondthm} follows by a standard compactness argument based on estimates \eqref{novacuum} and \eqref{stability2}. This is achieved by approximating the initial data by smooth functions satisfying the hypotheses of Theorem \ref{principalthm}. Then, from estimates \eqref{novacuum} and \eqref{stability2} we may find a subsequence $(\rho_n,\uu_n,\HH_n,\psi_n)$ of the solutions corresponding to the regularized initial data that converges strongly to a limit $(\rho, \uu,\HH,\psi)$, which turns out to be the desired solution. In particular, by the Sobolev embedding and Aubin--Lions lemma, $\rho_n\to \rho$ in $C([0,T];C(\Omega_\y))$ and $\psi_n\to \psi$ in $C([0,T];C(\Omega))$. Also, as will become clear in Section \ref{linearproblems}, where we prove several regularity estimates on the coordinate change, we also have that the sequence of Lagrangian transformations associated to $\uu_n$ converge strongly to the Lagrangian transformation associated to the limit velocity field $\uu$. With this, the coupling terms in the momentum equation and in the NLS converge in the sense of distributions to the corresponding ones in the limit. In turn, all the other nonlinearities from system \eqref{E2rho}-\eqref{E2psi} are accounted for, due to the strong convergence of the sequence; thus concluding that the limit $(\rho,\uu,\HH,\psi)$ is indeed a solution.

The rest of the paper is organized as follows. In Section \ref{linearproblems} we prove several estimates on the linear and nonlinear problems associated to equations \eqref{E2rho}, \eqref{E2H} and \eqref{E2psi}, where the velocity field $\uu$ is given; and also the linearisation \eqref{E2ulinear} of \eqref{E2u}. We also study the regularity of the Lagrangian coordinate associated to a given velocity field $\uu$ in order to deal with the Sobolev norms in different coordinates. In Section \ref{S3} we prove the local existence of strong solutions. Section \ref{S4} is devoted to the proof of the low order a priori estimates \eqref{novacuum} and \eqref{stability2}. In Section \ref{S5} we complete the necessary a priori estimates that allow us to extend the local solutions to global ones. At last, in Section \ref{S6} we state a weak stability result, that follows as a byproduct of the proof of the contractivity of the solution operator to the linearized system employed in the local result.

\section{Linear problems}\label{linearproblems}

As explained above, we construct a local solution of system \eqref{E2rho}-\eqref{E2psi} as a fixed point of the solution operator of the equation \eqref{E2ulinear}, where $\rho$, $\HH$ and $\psi$ are the solutions of equations \eqref{E2rho}, \eqref{E2H} and \eqref{E2psi} for a given velocity field $\uu$. 

For this reason we dedicate this section to the investigation of the linear and nonlinear problems, corresponding to equations \eqref{E2rho}, \eqref{E2H} and \eqref{E2psi}, where the velocity field $\uu$ is given; and also the linearization \eqref{E2ulinear} of \eqref{E2u}. We also study the regularity of the Lagrangian coordinate associated to a given velocity field $\uu$ in order to deal with the Sobolev norms in different coordinates.

In the following, we denote by $\phi = \phi(\cdot,\cdots, \cdot)$ a \textit{continuous} and \textit{non-decreasing} real function of its arguments that may change from line to line.

\subsection{The linearized parabolic problem}

Let us consider the following initial value problem
\begin{equation}
\begin{cases} \label{problemaparabolico}
\rho(\x,t)  \frac{\partial \uu}{\partial t}(\x,t) - (L_\rho \uu)(\x,t) = f(\x,t) & \text{ for } 0<t<T \text{ and } \x \in \Omega  \\
\uu(\x,0) = \uu_0(\x) & \text{ for } 0=t \text{ and } \x \in \Omega;
\end{cases}
\end{equation}
where $\Omega$ is the bidimensional torus, $T$ is a positive number, the density $\rho(\x,t)$ is a given function, $L_\rho$ is the associated Lamé operator given by
\begin{equation}
(L \uu)(x) = (L_\rho \uu)(x) = - \Div (\lambda(\rho)(\Div \uu)\text{Id} + \mu (\nabla\uu + (\nabla\uu)^\top)), \label{operatorL}
\end{equation}
and $\uu_0(\x)$ and $f(\x,t)$ lie on adequate Sobolev or Lebesgue spaces.

Note that this is a linear parabolic problem and can be solved by standard methods from the theory of linear parabolic equations, provided that the given density $\rho$ is strictly positive and sufficiently smooth. Moreover, we have the following a priori estimates on the solutions.
\begin{lemma}\label{highparabol}
Let $m \geq 1$ and assume that
		\begin{equation}
		\begin{cases} \label{condparabol2}
		\rho &\in C([0,T];H^{m}(\Omega))\cap F_{r,q}, \\
		f &\in L^2(0,T;H^{m-1}(\Omega;\mathbb{R}^2)), \text{ and } \\
		\uu_0 &\in H^{m}(\Omega;\mathbb{R}^2).
		\end{cases}
		\end{equation}
		for some $r\geq 2$ and $q>2$, where $F_{r,q}:=W^{1,\infty}(0,T; L^2(\Omega)) \cap C([0,T]; W^{1,q}(\Omega))$. Assume further that
		\[
           0<m(T),\qquad M(T)<\infty		
		\]
		 where, $m(t) := \min_{0\leq t' \leq t, x \in \Omega} \rho(\x,t')$ and $M(t) := \max_{0\leq t' \leq t, x \in \Omega} \rho(\x,t')$. 
		
		Then, \eqref{problemaparabolico} has a unique solution $\uu \in H^1(0,T;H^{m-1}(\Omega)) \cap L^2(0,T;H^{m+1})$ and we have the following estimate
		\begin{align}
		\int_{0}^{T} &(\Vert \uu_t(t') \Vert_{H^{m-1}}^2 + \Vert \uu(t') \Vert_{H^{m+1}}^2 )dt' \nonumber \\ &\leq \phi(T, m(t)^{-1}, \Vert \rho \Vert_{F_{r,q}\cap C([0,T];H^m)}) (\Vert f \Vert_{L^2(0,T;H^{m-1})}^2 + \Vert \uu_0 \Vert_{H^m}^2).\label{apriorimparabol}
		\end{align}
\end{lemma}

\begin{proof}	
	We proceed by induction in $m\geq 1$.
	
	Multiplying equation \eqref{problemaparabolico} by $\uu$, integrating by parts and using Young's inequality we get
	\begin{align*}  
	&\frac{1}{2}\frac{d}{dt}\int_\Omega \rho|\uu|^2 d\x + \int_\Omega (\mu |\nabla \mathbf{u}|^2+(\lambda(\rho)+\mu)(\Div\mathbf{u})^2) d\x\\
       &\qquad\leq C\left(\| f(t)\|_{L^2(\Omega)}^2 + (1+\|\rho_t\|_{L^\infty(0,T;L^r(\Omega))}^2 ) \int_\Omega |\uu|^2d\x\right) + \frac{\mu}{2}\int_\Omega |\nabla\uu|^2d\x,
	\end{align*}
	where we have used Gagliardo-Nirenberg inequality in order to estimate
	\[
	\| \uu\|_{L^{2r'}(\Omega)}^2\leq C_\varepsilon\|\uu\|_{L^2(\Omega)}^2 + \varepsilon \|\nabla\uu\|_{L^2(\Omega)}^2,
	\]
	with small enough $\varepsilon>0$.	 Then, using the uniform bounds away from zero for the density, Gronwall's inequality yields
	\begin{align}
	&\int_{\Omega} |\uu(\x,t)|^2d\x + \int_{0}^{t} \int_{\Omega} |\nabla \uu(\x,t') |^2d\x dt'  \nonumber \\ 
	&\qquad\leq \phi(t, m(t)^{-1}, M(t), \Vert \rho \Vert_{F_{r,q}})\Big( \int_{0}^{t}\int_{\Omega} |f(\x,t')|^2d\x dt' + \int_{\Omega} |\uu_0(\x)|^2d\x  \Big). \label{parabolest1}
	\end{align}
	
	Similarly, noting that
	\begin{align*}
	&\int_\Omega L_\rho\uu\cdot \uu_t d\x \\
	&\qquad= -\frac{1}{2}\frac{d}{dt}\int_\Omega (\mu |\nabla \mathbf{u}|^2+(\lambda(\rho)+\mu)(\Div\mathbf{u})^2) d\x +\int_\Omega \lambda(\rho)_t (\Div \uu)^2d\x,\nonumber\\
	&\qquad\leq -\frac{1}{2}\frac{d}{dt}\int_\Omega (\mu |\nabla \mathbf{u}|^2+(\lambda(\rho)+\mu)(\Div\mathbf{u})^2) d\x\\
	&\qquad\qquad\qquad\qquad\qquad\qquad+\|\lambda(\rho)_t\|_{L^\infty(0,T;L^r(\Omega))}\|\Div\uu(t)\|_{L^{2r'}(\Omega)}^2,
	\end{align*}
	multiplying \eqref{problemaparabolico} by $\uu_t$ and integrating by parts we have
	\begin{align*}
	&\int_\Omega \rho|\uu_t|^2 d\x + \frac{1}{2}\frac{d}{dt}\int_\Omega (\mu |\nabla \mathbf{u}|^2+(\lambda(\rho)+\mu)(\Div\mathbf{u})^2) d\x\\
	&\qquad\qquad\qquad\qquad\qquad\qquad\leq \varepsilon\|\uu\|_{H^2(\Omega)}^2 + \phi(\varepsilon^{-1},M(t),\|\rho_t\|_r)\|\nabla\uu\|_{L^2(\Omega)}^2,
	\end{align*}
	where we used again Gagliardo-Nirenberg inequality to estimate $\|\Div\uu\|_{L^{2r'}(\Omega)}^2$.
	
	Since $L_\rho$ is uniformly elliptic, one can apply Gagliardo-Nirenberg inequality to get\footnote{To see this, one may argue as follows. Let $\uu \in H^2(\Omega)$ and $f \in L^2(\Omega)$ be such that $L_\rho \uu + \uu = f$. If one uses the classic argument of difference quotients due to Nirenberg, one can show that $\mu \Vert \nabla^2 \uu \Vert_2^2 + \Vert \nabla u \Vert_{L^2}^2 \leq \Vert f \Vert_{L^2} \Vert \nabla^2 \uu \Vert_{L^2} + \Vert \nabla^2 \uu \Vert_{L^2} \Vert \nabla (\lambda(\rho)) \Vert_q  \Vert \operatorname{div} \uu \Vert_{L^{2q/(q-2)}}$. Hence, Gagliardo--Nirenberg and Young inequalities, and an elementary $H^1-$estimate assert that $\Vert \uu \Vert_{H^2} \leq \phi(\Vert  \nabla (\lambda(\rho)) \Vert_{W^{1,q}}) \cdot ( \Vert f \Vert_{L^2} + \Vert \uu \Vert_{L^2} )$, as desired. }
	\begin{align*}
	\Vert \uu \Vert_{H^2} &\leq \phi(\Vert \nabla \lambda(\rho) \Vert_{L^{q}} ) \cdot (\Vert L_\rho \uu \Vert_2 + \Vert \uu \Vert_2) \nonumber \\
					     &\leq \phi(\Vert \nabla \lambda(\rho) \Vert_{L^{q}} ) \cdot(\Vert  \rho \uu_t \Vert_2 + \Vert f \Vert_2 + \Vert \uu \Vert_2).\\
	\end{align*}
Thus, choosing $\varepsilon>0$ small enough,  Gronwall's inequality implies \eqref{apriorimparabol} for $m=1$ by virtue of the uniform bounds on the density.

	Now let $m\geq 2$ and assume that \eqref{apriorimparabol} holds for $1,2,...,m-1$. Then, we apply the operator $\nabla^{m-1}$ to equation \eqref{problemaparabolico} to obtain
	\begin{equation*}
	\begin{cases} 
	\rho (\nabla^{m-1} \uu_t) - (L_\rho \nabla^{m-1} \uu) = \nabla^{m-1} f+ \nabla^{m-2} (\Div (\nabla \lambda(\rho) (\Div \uu) \text{Id} )) \\
	\>\>\>\>\>\>\>\>\>\>\>\>\>\>\>\>\>\>\>\>\>\>\>\>\>\>\>\>\>\>\>\>\>\>\>\>\>\>\>\>\>\>\>\>\>\>\>\>\>\>\>\>\>\>\>\> - \nabla^{m-2} (\nabla \rho \otimes \uu_t) \\ \>\>\>\>\>\>\>\>\>\>\>\>\>\>\>\>\>\>\>\>\>\>\>\>\>\>\>\>\>\>\>\>\>\>\>\>\>\>\>\>\>\>\>\>\>\>\>\>\>\>\>\>\>\>\>\>\>\>\>\>\>\>\>\>\>\>\>\>\>\>\>\>\>\>\>\>\>\>\>\>\>\>\>\> \text{ for } 0<t<T \text{ and } \x \in \Omega,  \\
	(\nabla^{m-1}\uu)(\x,0) = \nabla^{m-1} \uu_0(\x) \\ \>\>\>\>\>\>\>\>\>\>\>\>\>\>\>\>\>\>\>\>\>\>\>\>\>\>\>\>\>\>\>\>\>\>\>\>\>\>\>\>\>\>\>\>\>\>\>\>\>\>\>\>\>\>\>\>\>\>\>\>\>\>\>\>\>\>\>\>\>\>\>\>\>\>\>\>\>\>\>\>\>\>\>\> \text{ for } 0=t \text{ and } \x \in \Omega.
	\end{cases}
	\end{equation*}
	From the case $m=1$, we get that
	\begin{align}
&\int_{0}^{t} ( \Vert \nabla^{m-1}\uu_t(t') \Vert_{L^2}^2 + \Vert \nabla^{m+1}\uu (t') \Vert_{L^2}^2 )   dt'   \nonumber \\ &\leq\phi(t, m(t)^{-1}, M(t), \Vert \rho \Vert_{F_{r,p}})  \Bigg( \int_{0}^{t}\int_{\Omega} |\nabla^{m-1}f(\x,t')|^2d\x dt' + \Vert \uu_0 \Vert_{H^m}^2  \nonumber \\
							&\qquad\qquad\qquad\qquad\qquad + \int_0^t\Vert \nabla^{m-1} ((\nabla \lambda(\rho(t))(\Div \uu(t)) \Vert_2^2dt' \nonumber\\
							&\qquad\qquad\qquad\qquad\qquad + \int_0^t \Vert  \nabla^{m-2} (\nabla \rho(t') \otimes \uu_t (t'))  \Vert_2^2dt' \Bigg). \label{estlemma}
	\end{align}
	To estimate the last two terms, let us divide into two cases cases. 
	
	i) If $m=2$, then
	\begin{align}
	&\Vert \nabla^{m-1} ((\nabla \lambda(\rho(t))(\Div \uu(t)) \Vert_2 \nonumber\\
	&\qquad\qquad\leq \Vert \nabla^2 \lambda(\rho(t)) \Vert_2 \Vert \Div \uu(t) \Vert_\infty + \Vert \nabla \lambda(\rho(t)) \Vert_q  \Vert \nabla^2 \uu(t) \Vert_p, \text{ and } \nonumber\\
	&\Vert  \nabla^{m-2} (\nabla \rho(t') \otimes \uu_t (t'))  \Vert_2 \leq \Vert \nabla \lambda(\rho(t')) \Vert_p \Vert \uu_t(t') \Vert_q,  \label{estlemma1}
	\end{align}
	where  $1/p + 1/q = 1/2$. Nonetheless,  Gagliardo--Nirenberg inequality implies
	\begin{align}
	\Vert \Div \uu(t) \Vert_\infty &\leq \varepsilon \Vert \nabla^3 \uu(t) \Vert_{2} + C_\varepsilon \Vert \nabla \uu (t) \Vert_2, \label{gagniren} \\
	\Vert \nabla^2 \uu(t) \Vert_p &\leq \varepsilon \Vert \nabla^3 \uu(t) \Vert_{2} + C_\varepsilon \Vert \nabla \uu (t) \Vert_2, \text{ and } \nonumber\\
	\Vert  \uu_t(t) \Vert_q &\leq \varepsilon \Vert \nabla \uu_t(t) \Vert_{2} + C_\varepsilon \Vert \uu_t (t) \Vert_2; \nonumber
	\end{align}
	thus, plugging \eqref{estlemma1} into \eqref{estlemma}, one can deduce the desired conclusion from the induction hypothesis.
	
	ii) Let us consider now the case $m\geq 3$. Once that $H^{m-1}$ and $H^{m-2} \cap L^\infty$ are algebras, we can show that 
	\begin{align*}
	\Vert \nabla^{m-1} ((\nabla \lambda(\rho)(t)(\Div \uu(t)) \Vert_2 &\leq \Vert  (\nabla \lambda(\rho(t)))(\Div \uu(t)) \Vert_{H^{m-1}} \\
	&\leq C_m \Vert \lambda(\rho(t)) \Vert_{H^m} \Vert \uu(t) \Vert_{H^{m} }, \text{ and} \\
	\Vert  \nabla^{m-2} (\nabla \rho(t') \otimes \uu_t (t'))  \Vert_2 &\leq  \Vert  \nabla \rho(t') \otimes \uu_t (t'))  \Vert_{H^{m-2}} \\
	&\leq C_m ( \Vert  \nabla \rho(t') \Vert_\infty \Vert \uu_t(t') \Vert_{H^{m-2}} \\ &\qquad\qquad + \Vert \rho (t') \Vert_{H^{m-2}}  \Vert \uu_t(t')  \Vert_{\infty} ).
	\end{align*}
	Thence, the result follows from the induction hypothesis and a Gagliardo--Nirenberg inequality similar to \eqref{gagniren}.
\end{proof}

\subsection{The continuity equation in terms of $\uu$}

We now consider the continuity equation \eqref{E2rho} in terms of a given velocity field $\uu$. Specifically we consider the problem: 
\begin{equation}
\begin{cases} \label{continuity}
\rho_t + \Div (\rho \uu) = 0 &\text{on } \Omega\times(0,T),\\
\rho(\x,0) = \rho_0(\x) &\text{on } \Omega\times\{ t=0 \},
\end{cases}
\end{equation}
where $\rho_0$ is \textit{positive} and belongs to $W^{1,q}(\Omega)$ for some $q>2$.

It is well known (see, e.g., the article of V.A. Solonnikov \cite{Solonnikov} and the references therein) that this problem has a unique solution given by the method of characteristics. In fact we have that
\begin{equation}
\rho(t,\Phi(t;\x)) = \rho_0(\x) \cdot \exp \Big\{ - \int_{0}^{t} (\Div \uu)\big(t', \Phi(t';\x)\big) dt'  \Big\}, \label{rhoeuler}
\end{equation} 
where $\Phi$ is the flow of $\uu$ given by
\begin{equation}
\begin{cases}
\frac{d\Phi}{dt}(t;\x)=\uu(\Phi(t;\x),t),\\
\Phi(0;\x)=\x.\label{flow(u)}
\end{cases}
\end{equation}

Moreover, $\rho$ is defined in the interval $[0,T]$ and satisfies
	\begin{align}
	&\min_{\x\in\overline{\Omega}}\rho_0(\x) \exp\left\{ -\int_0^t \|\Div\uu(t') \|_{L^\infty(\Omega)}dt'\right\} \nonumber\\
	&\qquad\qquad\qquad\qquad\qquad\leq \rho(t,\x)\leq \max_{\x\in\overline{\Omega}}\rho_0(\x)\exp\left\{ \int_0^t \|\Div\uu (t')\|_{L^\infty(\Omega)}dt'\right\},\label{uniformrho(u)}
	\end{align}
	for any $(t,\x)\in [0,T]\times \Omega$, provided that $\Div\uu\in L^1(0,T;L^\infty(\Omega))$.

Regarding the regularity of $\rho$ we have the following.

\begin{lemma} \label{higheu}
	Let $m\geq 1$ and assume that
	\begin{equation*}
	\begin{cases} 
	\uu \in L^2(0,T;H^{m+1}(\Omega)),&\text{with }\nabla \uu \in L^1(0,T;L^\infty(\Omega)), \\
	\rho_0 \in H^{m}(\Omega),&\text{with }0\leq \rho_0 <M,
	\end{cases}
	\end{equation*} 
	for some positive constant $M$.
	
	Then, regarding the solution of \eqref{continuity}, we have the estimate
	\begin{align}
	&\Vert \rho(t) \Vert_{H^m(\Omega)} \leq \phi(t,\|\uu\|_{L^2(0,T;H^{m}(\Omega))},\|\nabla\uu\|_{L^1(0,T;L^\infty(\Omega))},\|\rho_0\|_{H^m(\Omega)}, \nonumber\\ 
	&\qquad\qquad\qquad\qquad\qquad \|\uu\|_{L^2(0,T;W^{2,4}(\Omega))},M)\cdot(1+\|\uu\|_{L^2(0,T;H^{m+1})}), \label{Hmrho}
	\end{align}
	for all $0< t < T$. Furthermore, when $m=1$, the dependence of $\|\uu\|_{L^2(0,T;W^{2,4}(\Omega))}$ above may be dropped.
\end{lemma}

\begin{proof}
	Let $m\geq 1$. Then, applying the operator $\nabla^m$ to equation \eqref{continuity}, multiplying by $\nabla^m \rho$ and integrating by parts we have
	\begin{align}
	&\frac{1}{2}\frac{d}{dt}\int_\Omega |\nabla^m\rho|^2 dx \nonumber\\
	&\qquad= -\frac{1}{2}\int_\Omega  |\nabla^m\rho|^2\Div\uu d\x - \int_\Omega\nabla^{m-1}(\nabla\uu\cdot\nabla\rho+\rho\nabla\Div\uu)\cdot\nabla^m\rho d\x\nonumber\\
	&\qquad\leq C\|\nabla^{m-1}(\nabla\uu\cdot\nabla\rho+\rho\nabla\Div\uu)\|_{L^2(\Omega)}^2 \nonumber\\
	&\qquad\qquad\qquad+ C(1+\|\nabla\uu(t)\|_{L^\infty(\Omega)})\int_\Omega |\nabla^m\rho|^2 dx.\label{ineqrho(u)m}
	\end{align}
	
With this observation at hand, we proceed with the proof by induction in $m$.

So, let $m=1$. Then, we have that 
\begin{align*}
&\|\nabla\uu\cdot\nabla\rho+\rho\nabla\Div\uu\|_{L^2(\Omega)}\leq \|\rho\|_{L^\infty((0,T)\times\Omega)}\|\nabla^2 \uu(t)\|_{L^2(\Omega)}\\
&\qquad\qquad\qquad\qquad\qquad\qquad +\|\nabla\uu(t)\|_{L^\infty(\Omega)}\|\nabla\rho(t)\|_{L^2(\Omega)}.
\end{align*}
Hence, by virtue of \eqref{uniformrho(u)} and \eqref{ineqrho(u)m}, Gronwall's inequality yields the result for $m=1$.

For $m=2$, we have that
\begin{align*}
&\|\nabla(\nabla\uu\cdot\nabla\rho+\rho\nabla\Div\uu)\|_{L^2(\Omega)} \\
&\qquad\leq \|\nabla\uu(t)\|_{L^\infty(\Omega)}\|\nabla^2\rho\|_{L^2(\Omega)}+\|\rho(t)\|_{L^\infty(\Omega)}\|\nabla^3 \uu(t)\|_{L^2(\Omega)} \\
&\qquad\qquad\qquad+ \|\nabla\rho(t)\|_{L^4(\Omega)}\|\nabla^2 \uu(t)\|_{L^4(\Omega)}.
\end{align*}

On the other hand, applying the operator $\nabla$ to the continuity equation \eqref{E2rho}, multiplying the resulting equation by $p|\nabla\rho|^2\nabla\rho$ and integrating we have
\begin{align*}
\frac{d}{dt}\int_\Omega|\nabla\rho|^4d\x&=-3\int_\Omega |\nabla\rho|^4\Div\uu d\x - 4\int_\Omega |\nabla\rho|^{p-2}\nabla\rho\cdot\nabla\uu\cdot\nabla\rho d\x\\
&\qquad -p\int_\Omega \rho|\nabla\rho|^2\nabla\rho\cdot\nabla\Div\uu d\x,
\end{align*}
so that
\begin{equation}
\frac{d}{dt}\|\nabla\rho\|_{L^4(\Omega)} \leq C(\|\nabla\uu\|_{L^\infty(\Omega)}\|\nabla\rho\|_{L^4(\Omega)}+\|\nabla^2\uu\|_{L^4(\Omega)});
\end{equation}
which implies the estimate
\begin{equation}
  	\Vert \nabla \rho(t) \Vert_4 \leq C e^{ c\int_{0}^{t} \Vert \nabla \uu(s) \Vert_\infty ds} \Big( \Vert \nabla\rho_0 \Vert_4 + \Vert \rho_0 \Vert_\infty \int_{0}^{t} \Vert \nabla^2 \uu (s) \Vert_4 ds \Big), \label{rhoW14}
\end{equation}
for universal constants $C$ and $c$. Consequently, $m=2$ follows from Gronwall's inequality and the case $m=1$.

Finally, we recall that when $m\geq 3$ the space $H^{m-1}$ is an algebra so that
\begin{align*}
&\|\nabla^{m-1}(\nabla\uu\cdot\nabla\rho+\rho\nabla\Div\uu)\|_{L^2(\Omega)}\\
&\qquad\leq \|\uu(t)\|_{H^m(\Omega)}\|\rho(t)\|_{H^m(\Omega)} + \|\rho(t)\|_{H^{m-1}(\Omega)}\|\uu(t)\|_{H^{m+1}(\Omega)}.
\end{align*}
thus, once more, Gronwall's inequality and the case $m-1$ imply the result for any $m\geq 2$.
\end{proof}

\subsection{The magnetic field in terms of $\uu$}

We move on to considering the following problem for the magnetic field in terms of a given velocity field $\uu$.
\begin{equation}
\begin{cases} \label{H(u)}
\HH_t + \uu \cdot \nabla \HH - \HH\cdot \nabla\uu + \HH\Div \uu = \nu \Delta \HH, &\text{on }\Omega\times (0,T),\\
\Div \HH = 0,&\text{on }\Omega\times (0,T),\\
\HH(\x,0) = \HH_0(\x) &\text{on } \Omega\times \{ t=0 \},
\end{cases}
\end{equation}
with periodic boundary conditions (recall that $\Omega$ is the 2-dimensional torus). Note that this is a linear parabolic equation and can be solved by the standard Faedo-Galerkin method. Regarding the regularity of $\HH$ we have the following.
\begin{lemma}\label{lemmaH(u)}
Supose that $\HH_0 \in H^m(\Omega)$ and $\uu\in L^2(0,T;H^{m+1}(\Omega))$, for some $m\geq 1$. Then, \eqref{H(u)} has a unique solution $\HH\in L^2(0,T;H^{m+1}(\Omega))\cap C([0,T];H^m(\Omega))$. Moreover, we have the following a piori estimate
\begin{align}
&\Vert \HH(t) \Vert_{H^m(\Omega)}^2+\int_0^t\Vert \HH \Vert_{H^{m+1}}^2ds\nonumber\\
&\qquad\qquad\leq \phi(T,\Vert\uu\Vert_{L^2(0,T;H^m(\Omega))},\Vert \HH_0\Vert_{H^m}^2)(1+\Vert\uu\Vert_{L^2(0,T;H^{m+1}(\Omega))}),\label{estonH(u)}
\end{align}
for a.e. $t\in [0,T]$.
\end{lemma}

\begin{proof}
We omit the proofs of existence and uniqueness of solutions as they follow from standard procedures from the theory of linear parabolic equations. On the other hand estimate \eqref{estonH(u)} follows inductively by taking the $\alpha$ derivative of equation \eqref{H(u)}, for any multi-index $\alpha$ with $|\alpha|\leq m$, taking the inner product by $\partial^\alpha \HH$ and integrating by parts. After some standard manipulation, involving Young's inequality with $\varepsilon$, Gronwall's inequality yields the result.
\end{proof}

\subsection{The Lagrangian transformation}

Before we go into the resolvability of the NLS in terms of $\uu$, we have to establish some properties of the associated Lagrangian transformation. Since we deal with Sobolev norms in different coordinates it is very advantageous to have some information about the regularity of the change of variables.

Let us recall that the Lagrangian transformation $\Y(t,\x)=(t,\y(t,\x))$ can be formulated for any velocity field $\uu$ through the identity
\begin{equation}
\y(t,\Phi(t;\x))=\y_0(\x),\qquad \x\in\Omega\label{ydef}
\end{equation}
where $\y_0$ is a given diffeomorphism and $\Phi=\Phi(t;\x)$ is the flux associated to the velocity field of the fluid, satisfying the the ODE
\begin{equation}
\begin{cases}
\frac{d\Phi}{dt}(t;\x)=\uu(t,\Phi(t;\x)),\\
\Phi(0;\x)=\x.
\end{cases}
\end{equation}

In this opportunity, regarding the SW-LW Interactions model, we choose
\[
\y_0(\x)=\x,
\]

Note that if we denote $\tilde{\J}_\y(t,\x):=\J_\y(t,\Phi(t,\x)))=|\det \frac{\partial \y}{\partial \z}(t,\Phi(t,\x))|$ then using \eqref{ydef} along with Liouville's formula for the determinant applied to the gradient of $\Phi$ we have that
\begin{equation}
\frac{d}{dt}\tilde{\J}_\y(t,\x) = -\Div \uu(t,\Phi(t,\x)) \tilde{\J}_\y(t,\x);\label{liouvilleJ_y}
\end{equation}
and from the continuity equation \eqref{continuity}, a straightforward calculation shows that
\begin{equation}
\frac{d}{dt}\left[\frac{\tilde{J}_{\y}(t)}{\rho(t,\Phi(t;\x))}\right]=0.\label{J/rho=const}
\end{equation}
In other words, the function $\J_\y/\rho$ is constant along particle paths.

Taking into account our particular choice for $\y_0$ we conclude that
\begin{equation}
(\max_{\x\in\Omega}\rho_0(\x))^{-1}\leq \frac{\J_\y(t,\x)}{\rho(t,\x)}\leq (\min_{\x\in\Omega}\rho_0(\x))^{-1},\label{J/rho}
\end{equation}
and this holds for all $(\x,t)\in \Omega\times [0,\infty)$. In particular, we see that, $\y(t,\cdot)$ is a diffeomorphism from $\Omega$ onto $\Omega_\y$, for any $t\geq 0$, as long as the density is positive and finite. Let us recall that that, in fact, $\Omega_\y=\Omega$ is the 2-dimensional torus and that $\y(t,\cdot)$ is a change of variables from $\Omega$ onto itself. However we distinguish the domain of the Lagrangian coordinate with the sub-index $\y$ as we will deal with Sobolev norms in the different coordinates.

Moreover, defining the \emph{deformation gradient} $\E$ as $\E(\x,t):=\nabla_{\x}\y(\x,t)$, it can be shown that its entries $E_{ij}$ satisfy the following
\begin{equation}
\begin{cases}
\frac{d}{dt} E_{ij}(t,\x)+\sum_{k=1}^n\frac{\partial E_{ij}}{\partial x_k}(t,\x) u_k(t,\x) + \sum_{k=1}^nE_{ik}(t,\x)\frac{\partial u_k}{\partial x_j}(t,\x) = 0\\
E_{i,j}(t,0)=\delta_{ij}.
\end{cases}\label{defgrad}
\end{equation}
where $\uu=(u_1,...,u_n)$.

Indeed, this follows by taking the gradient in \eqref{ydef} and differentiating the resulting equation with respect to $t$.

Multiplying \eqref{defgrad} by $2m|E_{ij}|^{2m-2}E_{ij}$, for any given integer $m>2$, and integrating by parts we get
\begin{align*}
\frac{d}{dt}\int_\Omega |E_{ij}|^{2m}d\x&=\int_\Omega |E_{ij}|^{2m}\Div \uu d\x - 2m\int_\Omega |E_{ij}|^{2m-2}E_{ij}\sum_{k=1}^n E_{ik}\frac{\partial u_k}{\partial x_j}d\x &\\
          &\leq C ||\nabla \uu (t)||_{L^\infty(\Omega)}\int_\Omega |\E|^{2m}d\x,\label{ELp}
\end{align*}
where $C>0$ depends only on on $m$. Taking the sum over $1\leq i,j\leq 2$ we obtain
\begin{equation}
\frac{d}{dt}\int_\Omega |\E|^{2m}d\x\leq C ||\nabla \uu (t)||_{L^\infty(\Omega)}\int_\Omega |\E|^{2m}d\x.\label{ELpeq}
\end{equation}

In fact, realizing that \eqref{defgrad} is a transport equation we can obtain, through the method of characteristics, the following $L^\infty$ estimate
\begin{equation}
\|\E(t) \|_{L^\infty(\Omega)}\leq \exp\left\{ \int_0^t \|\nabla\uu(s)\|_{L^\infty(\Omega)} ds \right\}.\label{LinftyE}
\end{equation}

Furthermore, through a similar reasoning to that of the proof of lemma \ref{higheu} applied to equation \eqref{defgrad}, we can prove the following.
\begin{lemma} \label{regE}
	Let $m\geq 1$ and assume that $\uu \in L^2(0,T;H^{m+1}(\Omega))$, with $\nabla \uu \in L^1(0,T;L^\infty(\Omega))$.
	
	Then, $\mathbf{E}\in L^\infty(0,T;H^m(\Omega))$ and we have the estimate
	\begin{align}
	\Vert \mathbf{E}(t) \Vert_{H^m(\Omega)} \leq \phi(T,\|\uu\|_{L^1(0,T;H^{m+1}(\Omega))},\|\nabla\uu\|_{L^1(0,T;L^\infty(\Omega))}),
	\end{align}
	for all $0< t < T$.
\end{lemma}

We omit the proof as it goes by the same lines as the proof of lemma \ref{higheu}, once that we realize that the structure of the equations \eqref{continuity} and \eqref{defgrad} is very similar. 

To conclude this subsection, let us also analyse the inverse deformation gradient $\mathbf{B}(t,\y):=\frac{\partial \x}{\partial \y}(t,\y)$, where $\x(t,\y)$ is the inverse of the Lagrangian coordinate.

From the relation $\x(t,\y(t,\z))=\z$ and using the definition of the Lagrangian coordinate it can be shown that $\mathbf{B}$ satisfies the equation
\begin{equation}
\begin{cases}
\frac{d}{dt}\mathbf{B}(t,\y)=\nabla_\x\uu(t,\x(t,\y)) \mathbf{B}(t,\y), &(t,\y)\in .(0,T)\times \Omega_\y\\
\mathbf{B}(0,\y)= Id,
\end{cases}\label{eqB(u)}
\end{equation} 
where $Id$ is the $2\times 2$ identity matrix. Indeed, it suffices to realize that relation \eqref{ydef} implies that $\x(t,\y)=\Phi(t,\y)$.

Using this equation it can be easily shown that
\begin{equation}
\| \mathbf{B}(t)\|_{L^\infty(\Omega_\y)}\leq C\exp\left( \int_0^t \|\nabla_\x \uu(s)\|_{L^\infty(\Omega)}ds \right),\label{LinftyB(u)}
\end{equation}
for all $0\leq t\leq T$, provided that $\nabla\uu\in L^1(0,T;L^\infty(\Omega))$. Moreover we have the following result regarding the regularity of $\mathbf{B}$.
\begin{lemma}\label{regB}
	Let $m\geq 1$ and assume that $\uu \in L^2(0,T;H^{m+1}(\Omega))$, with $\nabla \uu \in L^1(0,T;L^\infty(\Omega))$.
	
	Then, $\mathbf{B}\in L^\infty(0,T;H^m(\Omega_\y))$ and we have the estimate
	\begin{align}
	\Vert \mathbf{B}(t) \Vert_{H^m(\Omega_y)} \leq \phi(T,\|\uu\|_{L^2(0,T;H^{m+1}(\Omega))},\|\nabla\uu\|_{L^1(0,T;L^\infty(\Omega))}),\label{aprioriB(u)}
	\end{align}
	for all $0< t < T$.
\end{lemma}

In order to prove this result we need the following observation about $L^q$ and Sobolev norms in different coordinates.
\begin{remark}\label{Lp_yLp_x}
(i) From equation \eqref{liouvilleJ_y} we have that
\begin{equation}
\| J_\y(t)\|_{L^\infty(\Omega)}\leq C\exp\left( \int_0^t \|\Div\uu(s)\|_{L^\infty} ds\right).
\end{equation}
Therefore, for any function $f\in L^q(\Omega)$ we have
\[
\| f(\x(t,\cdot))\|_{L^q(\Omega_\y)} = \| |f(\cdot)|^q\J_\y(t)\|_{L^1(\Omega)}^{1/q}\leq \phi(\|\nabla\uu\|_{L^1(0,T;L^\infty(\Omega))}) \|f\|_{L^q(\Omega)}. 
\]

(ii) On the other hand, a straightforward combinatorial argument (which can be made rigorous using induction) shows that there are universal polynomials $Q_m^k$ of degree $k$, $k=1,...,m$, such that, for any smooth enough function $f$
\begin{align}
  |\nabla_\y^m & f(\x(t,\y))|\leq \nonumber \\ &\sum_{k=1}^m |\nabla_\x^k f(\x(t,\y))|Q_m^k\big(|\mathbf{B}(t,\y)|,|\nabla_\y\mathbf{B}(t,\y)|,...,|\nabla^{m-k}\mathbf{B}(t,\y)|\big). \label{chainrule}
\end{align}
In fact, $Q_m^1 (z_0,z_1,...,z_{m-1})=z_{m-1}$. In particular, if $f\in H^m(\Omega)$, Gagliardo-Nirenberg inequality implies that,
\begin{align*}
&\| \nabla_\y^m f(\x(t,\cdot)\|_{L^2(\Omega)}\\
&\leq \phi(\|\nabla\uu\|_{L^1(0,T;L^\infty(\Omega))})\Big(\|\mathbf{B}\|_{L^\infty(\Omega)}^m\|\nabla_\x^m f(t)\|_{L^2(\Omega)} +\\
&\qquad+ \|\nabla f(t)\|_{L^\infty(\Omega)}\|\nabla_\y^{m-1}\mathbf{B}\|_{L^2(\Omega)} + \|f(t)\|_{H^{m}(\Omega)}(1 + \|\mathbf{B}(t)\|_{H^{m-1}(\Omega_\y)}^m)\Big),
\end{align*}
where we have also used part (i) above.

(iii) When $m=2$ we may refine this estimate as
\begin{align*}
&\| \nabla_\y^2 f(\x(t,\cdot)\|_{L^2(\Omega)}\\
&\leq \phi(\|\nabla\uu\|_{L^1(0,T;L^\infty(\Omega))})\Big(\|\mathbf{B}\|_{L^\infty(\Omega)}^m\|\nabla_\x^2 f(t)\|_{L^2(\Omega)} ++ \|\nabla f(t)\|_{L^4(\Omega)}\|\nabla_\y\mathbf{B}\|_{L^4(\Omega)} \Big),
\end{align*}
\end{remark}

\begin{proof}[Proof of lemma \ref{regB}]
Applying the operator $\nabla_\y^m$ to equation \eqref{eqB(u)}, multiplying by $\nabla^m\mathbf{B}$, integrating over $\Omega_\y$ and using Young's inequality we have
\begin{equation}
\frac{d}{dt}\int_{\Omega_\y} |\nabla_\y^m\mathbf{B}(t,\y)|^2 d\y \leq \| \nabla_\y^m [\nabla_x \uu(t,\x(t,\cdot))\cdot \mathbf{B}(t,\cdot)] \|_{L^2(\Omega_\y)}^2 + \| \nabla_\y^m\mathbf{B}(t) \|_{L^2(\Omega_\y)}^2.\label{regBprelim}
\end{equation}

With this observation at hand, we proceed with the proof of \eqref{aprioriB(u)} by induction in $m$. 

Note that, for $m=1$,
\begin{align*}
&\| \nabla_\y [\nabla_x \uu(t,\x(t,\cdot))\cdot \mathbf{B}(t,\cdot)] \|_{L^2(\Omega_\y)}\\
&\leq \|\mathbf{B}(t)\|_{L^\infty(\Omega_\y)}^2\| \nabla_\x^2\uu(t,\x(t,\cdot)) \|_{L^2(\Omega_\y)}+\|\nabla_x\uu(t)\|_{L^\infty(\Omega)}\|\nabla_\y\mathbf{B}(t)\|_{L^2(\Omega_\y)},
\end{align*}
so that, using remark \ref{Lp_yLp_x} and inequality \eqref{regBprelim}, Gronwall's inequality yields the result.

For $m=2$ we see that
\begin{align*}
&\| \nabla_\y^2 [\nabla_x \uu(t,\x(t,\cdot)) \mathbf{B}(t,\cdot)\|_{L^2(\Omega_\y)}\\
&\leq \|\mathbf{B}(t)\|_{L^\infty(\Omega_\y)}^3\| \nabla_\x^3\uu(t,\x(t,\cdot)) \|_{L^2(\Omega_\y)}+\|\nabla_\x\uu(t,\x(t,\cdot))\|_{L^\infty(\Omega_\y)}\|\nabla_\y^2\mathbf{B}\|_{L^2(\Omega_\y)}\\
&\qquad\qquad+2\|\nabla_x^2\uu(t,\x(t,\cdot))\|_{L^4(\Omega_\y)}\|\mathbf{B}\|_{L^\infty(\Omega_\y)}\|\nabla_\y\mathbf{B}(t)\|_{L^4(\Omega_\y)}\\
&\leq \|\mathbf{B}(t)\|_{L^\infty(\Omega_\y)}^3\| \nabla_\x^3\uu(t,\x(t,\cdot)) \|_{L^2(\Omega_\y)}+\|\nabla_\x\uu(t,\x(t,\cdot))\|_{L^\infty(\Omega_\y)}\|\nabla_\y^2\mathbf{B}\|_{L^2(\Omega_\y)}\\
&\qquad+\phi(\|\nabla\uu\|_{L^1(0,T;L^\infty(\Omega))})\|\uu(t)\|_{H^3(\Omega)}\|\mathbf{B}\|_{L^\infty(\Omega_\y)}\|\mathbf{B}(t)\|_{H^2(\Omega_\y)}.
\end{align*}
Thus, plugging this inequality in \eqref{regBprelim} and using \eqref{LinftyB(u)} together with the case $m=1$ and Gronwall's inequality we conclude the result for $m=2$.

When $m\geq 3$ we use the fact that the space $H^m$ is an algebra so that
\begin{equation}
\| \nabla_\y^m [\nabla_x \uu(t,\x(t,\cdot))\cdot \mathbf{B}(t,\cdot)] \|_{L^2(\Omega_\y)}\leq \|\nabla_x \uu(t,\x(t,\cdot))\|_{H^m(\Omega_\y)}\|\mathbf{B}(t)\|_{H^m(\Omega_\y)}.\label{ineqB1}
\end{equation}

Since in this case $\|\nabla_\x^2 \uu\|_{L^\infty(\Omega)}\leq C \|\uu\|_{H^{m+1}(\Omega)}$, then using part (ii) of remark \ref{Lp_yLp_x} we have that
\begin{align}
&\|\nabla_x \uu(t,\x(t,\cdot))\|_{H^m(\Omega_\y)}\nonumber\\
&\leq \phi(\|\nabla\uu\|_{L^1(0,T;L^\infty(\Omega))},\|\mathbf{B}\|_{L^\infty((0,T)\times\Omega_\y)},\|\mathbf{B}\|_{L^\infty(0,T;H^{m-1}(\Omega_\y))})\|\uu(t)\|_{H^{m+1}(\Omega)}.\label{ineqB2}
\end{align}

Gathering \eqref{ineqB1} and \eqref{ineqB2} in \eqref{regBprelim}, the result follows by Gronwall's inequality and the induction hypothesis.
\end{proof}

As a corollary of lemmas \ref{regE} and \ref{regB}, we have the following equivalence between Sobolev norms in Eulearean and Lagrangian coordinates.
\begin{corollary}\label{equivnorms}
Let $m\geq 1$ and assume that $\uu \in L^2(0,T;H^{m}(\Omega))$, with $\nabla \uu \in L^1(0,T;L^\infty(\Omega))$. Then, there are positive constants $C_1$ and $C_2$ depending on $T$, $\|\uu\|_{L^2(0,T;H^{m}(\Omega))}$ and $\|\nabla\uu\|_{L^1(0,T;L^\infty(\Omega))}$ such that
\begin{equation}
\| f(\x(t,\cdot))\|_{H^m(\Omega_\y)}\leq C_1\|f\|_{H^m(\Omega)}\leq C_2\| f(\x(t,\cdot))\|_{H^m(\Omega_\y)},\label{equivnormseq}
\end{equation}
for any $f\in H^m(\Omega)$.

When $m=2$, the conclusion remains valid, if $\nabla^2\uu\in L^2(0,T;L^4(\Omega))$.  
\end{corollary}

\begin{proof}
When $m\geq 3$ we may use Remark \ref{Lp_yLp_x} and the fact that $\|\nabla f\|_{L^\infty(\Omega)}\leq C\|\nabla f\|_{H^m(\Omega)}$. For the case $m=2$ we use part (iii) of Remark \ref{Lp_yLp_x} instead. Here we have to deal with the $L^4$ norm of $\nabla \mathbf{B}$. However, applying the operator $\nabla_\y$ to equation \eqref{eqB(u)} multiplying by $|\nabla_\y \mathbf{B}|^2\nabla_\y\mathbf{B}$ and integrating, we see that
\begin{align*}
&\frac{d}{dt}\|\nabla\mathbf{B}(t)\|_{L^4(\Omega_\y)}^4 \leq C\phi(\|\nabla_\x\uu\|_{L^1(0,T;L^\infty(\Omega))})\Big(\|\mathbf{B}\|_{L^\infty(\Omega_\y))}^2\|\nabla_\x^2\uu\|_{L^4(\Omega)}\|\nabla_\y\mathbf{B}\|_{L^4(\Omega_\y)} \\
&\qquad\qquad\qquad\qquad\qquad\qquad+\|\nabla_\x\uu\|_{L^\infty(\Omega)}\|\nabla_\y\mathbf{B}\|_{L^4(\Omega_\y)}^4 \Big),
\end{align*}
thus, obtaining the following bound, through Gronwall's inequality
\[
\|\nabla_\y\mathbf{B}(t)\|_{L^4(\Omega_\y)}\leq \phi(\|\nabla_\x\uu\|_{L^1(0,T;L^\infty(\Omega))},\|\nabla\uu\|_{L^2(0,T;L^4(\Omega))}).
\]

This takes care of the first inequality in \eqref{equivnormseq}. The second inequality follows by similar considerations involving the deformation gradient $\mathbf{E}$ instead of $\mathbf{B}$.
\end{proof}

\subsection{The NLS in terms of $\uu$}

Finally, let us analyse the following NLS equation in the Lagrangian coordinates $(t,\y)\in [0,T]\times\Omega_\y$ associated to the given velocity field $\uu$
\begin{equation} \label{schr}
\begin{cases} 
i\psi_t + \Delta_\y \psi = |\psi|^2\psi +g(v) h'(|\psi|^2) \psi &\text{on } (0,T)\times\Omega_\y;\\
\psi = \psi_0 &\text{on } \{t=0\}\times\Omega_\y,
\end{cases}
\end{equation}
with periodic boundary conditions, where $g$ and $h$ are as in \eqref{assumptiongh} and $v$ is the specific volume, which is related $\uu$ through the the relation
\[v(\y(\x,t),t) = \frac{1}{\rho(\x,t)};\]
being $\rho$ the solution of \eqref{continuity}.

\begin{lemma}\label{existschr}
Let $m\geq 2$ and assume that
	\begin{equation*}
	\begin{cases}
	\rho_0 \in H^m(\Omega), &\text{with } 0<\rho_0 <\infty \\
	\uu  \in  L^2(0,T;H^{m+1}(\Omega)), &\text{with } \nabla \uu\in L^1(0,T;L^\infty(\Omega))
	\end{cases}
	\end{equation*}
so that, according to lemma \ref{higheu}, $\rho \in L^\infty(0,T;H^m(\Omega))$.

	Then, there is a unique solution $\psi \in C([0,T);H^m(\Omega_\y)) \cap C^1([0,T); H^{m-2}(\Omega_\y))$ of \eqref{schr}.	
	
    Moreover, we have the following a priori estimate
    \begin{align}
	&\Vert \psi(t) \Vert_{H^m(\Omega_\y)} \nonumber\\
	&\hspace{2mm}\leq \phi\big(T, \Vert \psi_0 \Vert_{H^m}, \Vert \rho_0 \Vert_{H^m}, (\min \rho_0)^{-1},\Vert \uu \Vert_{L^2(0,T;H^{m}(\Omega))},\|\nabla\uu\|_{L^1(0,T;L^\infty(\Omega))}  \big) \nonumber\\
	&\qquad\qquad\qquad\qquad\qquad\qquad\cdot(1+\|\uu\|_{L^2(0,T;H^{m+1}(\Omega))}) \label{aprioripsi(u)}
	\end{align}
    for all $0\leq t\leq T$.
\end{lemma}

As it is well known, provided that $\psi$ is sufficiently regular (for instance $\psi \in C^1([0,T];L^2(\Omega_\y)) \cap C([0,T];H^2(\Omega_\y))$), \eqref{schr} is equivalent to the integral equation given by the Duhamel formula
\begin{equation} \label{schr1}
\psi(t) = S(t) \psi_0 + \frac{1}{i} \int_0^t S(t-t') \Big( |\psi(t')|^2\psi(t') + g(v(t'))h'(|\psi(t')|^2)\psi(t')  \Big) dt',
\end{equation}
where we have denoted by $S(t) = \exp \Big( \frac{1}{i}\Delta_\y t \Big)$ the group of isometries associated to the linear Schr\"{o}dinger equation. One can also introduce the momentary notation
\begin{equation}
F(t,\psi) = \frac{1}{i} \Big( |\psi|^2\psi + g(v)h'(|\psi|^2)\psi \Big), \label{defF}
\end{equation}
so to make \eqref{schr1} look more like an ordinary differential equation solution:
$$\psi(t) = S(t)\psi_0 + \int_{0}^{t}S(t-t')F\big(t',\psi(t')\big)dt'.$$

Lemma \ref{existschr} will follow from the following classical result.

\begin{lemma} \label{pseudosegal}
	Let $H$ be a complex Hilbert space and $A : D(A) \subset H \rightarrow H$ a nonnegative self-adjoint operator. For some $s \geq 1$ and $T>0$, let also $F : [0,T] \times D(A^s) \rightarrow D(A^s)$ be a continuous transformation (where $D(A)$ is endowed with its graph norm), which is bounded and Lipschitz continuous on its second variable on bounded sets of $[0,T] \times D(A^s)$.
	
	Then, if $S(t) = \exp \{ - i A  t  \}$ is the group of isometries associated to $-iA$ and $\psi_0 \in D(A^s)$, there exists a $T^*>0$ and a unique and maximal solution $\psi \in C([0,T^*);D(A^s)) \cap C^1([0,T^*);D(A^{s-1}))$ of the integral equation
	$$\psi(t) = S(t)\psi_0 + \int_{0}^{t}S(t-t')F\big(t',\psi(t')\big)dt'.$$
	Moreover, one has the following alternative: either $T^* = T$ and $\psi$ may be extended to $t = T$; or $\Vert \psi(t) \Vert_{D(A^s)} \rightarrow \infty$ as $t \rightarrow T^*$.
\end{lemma}
\begin{proof}
	We have to show that the successive approximation sequence
	\begin{align*}
	\psi^{(0)}(t) &= S(t)\psi_0,\\
	\psi^{(n)}(t) &= S(t)\psi_0 + \int_0^t F\big(t',\psi^{(n-1)}(t')\big)dt'
	\end{align*}
	converges in $C([0,T');D(A^s)$ for a sufficiently small $T'$. This is elementary, as we briefly sketch below. 
	
	Let $X$ be the complete metric space of the $\psi : [0,T'] \rightarrow D(A^s)$ whose range lie on the ``tube'' $K = \{ S(t)\psi_0 + y; 0\leq t\leq T \text{ and } \Vert y \Vert_{D(A^s)} \leq \delta  \}$, where $\delta > 0$ is arbitrary and $T'>0$ will be chosen later. If $M > 0$ is such that $\Vert F(t,\psi) \Vert_{D(A^s)} \leq M$ for $0 \leq t \leq T$ and $\psi \in K$, pick $T'$ being the least between $T$ and $\delta/M$, so that the sequence $\{ \psi^{(n)} \}$ given above is entirely in $X$ (here we used that $I + A^s$ commutes with $S(t)$).
	
	Now the rest of the argument reminisces the proof of the Cauchy--Lipschitz--Picard theorem. If
	\begin{align*}
	\delta^{(0)}(t) &= \Vert u^{(0)}(t) \Vert_{D(A^s)},\\
	\delta^{(n)}(t) &= \Vert u^{(n)}(t) - u^{(n-1)}(t) \Vert_{D(A^s)},
	\end{align*}
	we can show by induction that
	\begin{align*}
	\delta^{(0)}(t) &\leq Mt, \text { and }\\
	\delta^{(n)}(t) &\leq M \frac{L^n t^n}{n!},
	\end{align*}
	so that $\operatorname{Sup}_t \sum \delta^{n} (t) < \infty$. By the Weierstrass M-test, $\psi^{(n)}$ converges uniformly in $C([0,T'], D(A^s))$ to a $\psi$, which is \textit{a fortiori} solution of the integral equation. Observe that a similar reasoning provides the uniqueness of solutions. 
	
	Regarding the blow up criteria we can argue as follows: if, by absurd, $T^* < T$ and $\Vert \psi(t) \Vert_{D(A^s)}$ remained bounded, one could apply the same argument above at initial time $T^* - \varepsilon$, for a small $\varepsilon > 0$, to obtain a solution defined at $[0,T'+\varepsilon]$; this is due to the fact that the lifespan of the local solution above depends only how large is $F$ on tubes like $K$. This contradiction settles the proof.
\end{proof}

Observe that, if $H = L^2(\Omega_\y)$ and $A : D(A) \subset H \rightarrow H$ is given by $Au = - \Delta u$ with $D(A) = H^2(\Omega_\y)$, then $A$ is nonnegative and self-adjoint; and for any $s \geq 0$, $D(A^s) = H^{2s}(\Omega)$ (with equivalent norms). In particular, if $s = m/2$, $D(A^s) = H^m$.
	
	Since $m \geq 2$, $H^m$ is an algebra and then the Sobolev inequalities and the equivalence of the Sobolev norms in Eulerean and Lagrangian coordinates established in the last subsection imply that $F$  given by \eqref{defF} is in accordance with the hypotheses of lemma \ref{pseudosegal}. Thence, the preceding abstract proposition guarantees the existence and uniqueness of local solutions. To conclude the proof of lemma \ref{existschr}, we have to show that the a priori estimate \eqref{aprioripsi(u)} holds so that the blow up criterion of lemma \ref{pseudosegal} is met and the local solutions are actually defined in the whole $[0,T]$. 

To prove \eqref{aprioripsi(u)} we proceed by induction, wherein the case $m=2$, contained in the following lemma, is an adaptation of an estimate on the Nonlinear Schr\"{o}dinger equation due to Brezis and Gallouet \cite{BrGa}.

\begin{lemma}\label{H1-H2(u)}
Let $\psi$ be a solution of equation \eqref{schr}. Then

(i) \begin{align}
&\int_{\Omega_\y}\left( \frac{1}{2}|\nabla_\y\psi(t)|^2 + \frac{1}{4}|\psi(t)|^4 +\alpha g(v)h(|\psi|^2) \right)d\y\nonumber \\
 &\qquad\qquad\qquad\qquad\leq \phi(T,\|\psi_0\|_{H^1(\Omega_\y)},(\min_{\x\in\Omega}\rho_0)^{-1},\|\Div \uu\|_{L^2(0,T;L^2(\Omega))}).\label{H1psi(u)}
\end{align}

(ii)\begin{align}
&\int_{\Omega_\y} \big( |\psi_t|^2+|\Delta_\y \psi|^2 \big) d\y\nonumber\\
&\qquad\qquad\qquad\qquad\leq \phi(T,\|\psi_0\|_{H^2(\Omega_\y)},(\min_{\x\in\Omega}\rho_0)^{-1},\|\Div \uu\|_{L^2(0,T;L^2(\Omega))}).\label{H2psi(u)}
\end{align}
In particular, for a.e. $t\in [0,T]$
\begin{equation}
\max_{\y\in\Omega_\y}|\psi(t)|\leq \phi(T,\|\psi_0\|_{H^2(\Omega_\y)},(\min_{\x\in\Omega}\rho_0)^{-1},\|\Div \uu\|_{L^2(0,T;L^2(\Omega))}).\label{uniformpsi(u)}
\end{equation}
\end{lemma}

\begin{proof}
Multiply \eqref{schr} by $\overline{\psi_t}$ (the complex conjugate of $\psi_t$), take real part and integrate integrate by parts to obtain
	\begin{align}
	\frac{d}{dt} &\int_{\Omega_\y} \Big\{ \frac{1}{2}|\nabla_\y \psi(\y,t)|^2 + \frac{1}{4} |\psi(\y,t) |^4 \Big\} d\y\nonumber \\
	&= - \int_{\Omega_\y} g(v(\y,t)) \frac{\partial}{\partial t} h (|\psi(\y,t)|^2) d\y\nonumber \\
	&= - \frac{d}{dt}\int_{\Omega_\y} g(v(\y,t)) h (|\psi(\y,t)|^2) d\y  + \int_{\Omega_\y}   h (|\psi(\y,t)|^2) g(v(\y,t))_t d\y.\label{H1psi(u)prelim}
	\end{align}
	
Note that from the definition of $\mathbf{Y}$ we have the conversion formula between Eulerian and Lagrangian coordinates:
\[
 \beta(t,\mathbf{y})_t=(\beta\circ\mathbf{Y}(t,\mathbf{x}))_t+\mathbf{u}\cdot\nabla_{\mathbf{x}}(\beta\circ\mathbf{Y}(t,\mathbf{x})),
\]
or synthetically,
\[
 \beta(t,\mathbf{y})_t=\beta_t(t,\mathbf{x})+\mathbf{u}\cdot\nabla_{\mathbf{x}}\beta(t,\mathbf{x}).
\]
Then, recalling \eqref{assumptiongh} and using the continuity equation \eqref{continuity} together with \eqref{J/rho} we have
\begin{align*}
\int_0^t \int_{\Omega_\y} |g'(v) v_t|^2 d\y ds &= \int_0^t \int_{\Omega} \big|g'(1/\rho) \big( (1/\rho)_t + \uu \cdot \nabla(1/\rho) \big)\big|^2 \J_\y d\x ds\\
   &= \int_0^t \int_{\Omega} \left|\frac{g'(1/\rho)}{\rho} \big( (\log \rho)_t + \uu \cdot \nabla(\log \rho)\big)\right|^2 \J_\y  d\x ds\\
   &= \int_0^t \int_{\Omega} \frac{g'(1/\rho)^2}{\rho} |\Div \uu |^2 \frac{\J_\y}{\rho} d\x ds\\
   &\leq C(\min_{\x\in\Omega}\rho_0)^{-1}\int_0^t\int_\Omega |\Div \uu |^2 d\x ds
\end{align*}

Putting this inequality together with \eqref{H1psi(u)prelim} we obtain \eqref{H1psi(u)} upon integrating in $t$.

Let us now prove \eqref{H2psi(u)}. First, in light of estimate \eqref{H1psi(u)} and using Sobolev's embedding for $H^1$ functions we see from equation \eqref{schr} that there is some $C>0$ that depending on $T$, $\|\psi_0\|_{H^1(\Omega_\y)}$, $(\min_{\x\in\Omega}\rho_0)^{-1}$ and $\|\Div \uu\|_{L^2(0,T;L^2(\Omega))}$ such that
\begin{equation}
\int_{\Omega_\y}|\psi_t|^2d\y - C\leq \int_{\Omega_\y}|\Delta\psi|^2d\y\leq \int_{\Omega_\y}|\psi_t|^2d\y +C.\label{psitDeltapsi(u)}
\end{equation}

Now, we differentiate equation \eqref{E2psi} with respect to $t$ to obtain
\begin{flalign*}
&i\psi_{tt}+\Delta_\y \psi_t = 2|\psi|^2 \psi_t+ \psi^2 \overline{\psi}_t  + \alpha \Big(g'(v)h'(|\psi|^2) v_t \\
&\hspace{40mm}+ g(v)h''(|\psi|^2)(|\psi|^2\psi_t + \psi^2\overline{\psi}_t)+ g(v)h'(|\psi|^2)\psi_t\Big),
\end{flalign*}
where the over line stands for complex conjugation. 

Multiplying this equation by $\overline{\psi}_t$, taking imaginary part, integrating and using Young's inequality we obtain
\begin{flalign}
&\int_{\Omega_\y}|\psi_t|^2d\y - \int_{\Omega_\y}|\psi_{0t}|^2d\y \nonumber\\
&\leq C\int_0^t \left(1+||\psi(t)||_{L^\infty(\Omega_\y)}^2\right) \int_{\Omega_\y}|\psi_t|^2d\y ds + C\int_0^t\int_{\Omega_\y}|g'(v)v_t|^2d\y ds.\label{estpsi(u)t1}&&
\end{flalign}

Let us recall Brezis and Gallouet's inequality (see \cite{BrGa}, {\em cf.} \cite{BrWa}) asserting that there is a constant $C>0$ depending only on $\Omega$ such that
\[
||\omega||_{L^\infty(\Omega)}\leq C\big(1+\sqrt{\log[1+||\omega||_{H^2(\Omega)}]}\big),
\]
for every $\omega\in H^2(\Omega)$ with $||\omega||_{H^1(\Omega)}\leq 1$.

Combining this result with \eqref{H1psi(u)} and \eqref{psitDeltapsi(u)} we see that
\begin{equation}
||\psi(t)||_{L^\infty(\Omega_\y)}\leq C\big(1+\sqrt{\log[1+||\psi_t(t)||_{L^2(\Omega_\y)}]}\big)
\end{equation}

With this, the first term on the right hand side of \eqref{estpsi(u)t1} is bounded by
\[
C\int_0^t \big(1+\log[1+||\psi_t(s)||_{L^2(\Omega_\y)}]\big)||\psi_t(s)||_{L^2(\Omega_\y)}^2ds.
\]

Coming back to \eqref{estpsi(u)t1} we have
\begin{equation}
||\psi_t(t)||_{L^2(\Omega_\y)}^2\leq C+C\int_0^t \big(1+\log[1+||\psi_t(s)||_{L^2(\Omega_\y)}^2]\big)||\psi_t(s)||_{L^2(\Omega_\y)}^2ds.
\end{equation}

Finally, defining
\[
G(t):=1+\int_0^t \big(1+\log[1+||\psi_t(s)||_{L^2(\Omega_\y)}^2]\big)||\psi_t(s)||_{L^2(\Omega_\y)}^2ds,
\]
we see that
\begin{align*}
\frac{d}{dt}G(t)&= \Big(1+\log[1+||\psi_t(t)||_{L^2(\Omega_\y)}^2]\big)||\psi_t(t)||_{L^2(\Omega_\y)}^2ds\\
&\leq CG(t)(1+\log[1+G(t)]).&&
\end{align*}

Thus,
\[
\frac{d}{dt}\log\Big[1+\log [1+G(t)]\Big]\leq C,
\]
which implies that
\begin{equation}
G(t)\leq \exp(\alpha \exp (\beta t)),
\end{equation}
for some $\alpha,\beta>0$ depending only on the initial data (note that they depend on the $H^2$ norm of $\psi_0$).

In order to conclude we point out that this estimate and \eqref{psitDeltapsi(u)} imply \eqref{H2psi(u)}, while \eqref{uniformpsi(u)} follows from the Sobolev embeddings.
\end{proof}

\begin{proof}[Proof of lemma \ref{existschr}]
Having proved the case $m=2$ we proceed by induction in $m\geq 3$. So, let us assume that the result holds for $m-1$, aiming to prove the case $m$. 

By \eqref{schr1} we have that
	\begin{align}
	\Vert \psi(t) \Vert_{H^m(\Omega_\y)} &\leq C  \Big(\Vert \psi_0 \Vert_{H^m(\Omega_\y)} + \int_0^t \big(\Vert \psi(t') \Vert_{L^\infty(\Omega_\y)}^2  \Vert \psi(t') \Vert_{H^m(\Omega_\y)} \nonumber \\ 
	&\qquad\qquad + \Vert g(v(t')) \Vert_{L^\infty(\Omega_\y)} \Vert h(|\psi(t')|^2)\psi(t')\Vert_{H^m(\Omega_\y)} \nonumber\\ 
	&\qquad\qquad + \Vert g(v(t')) \Vert_{H^m(\Omega_\y)} \Vert h'(|\psi(t')|^2)\psi(t')\Vert_{L^\infty(\Omega_\y)} \big)dt'\Big) \label{estschr}
	\end{align}
	
On the one hand, the $L^\infty$-norms appearing in this inequality are easily estimated using the hypotheses on the coupling functions as well as lemma \ref{H1-H2(u)}. On the other hand, by a similar reasoning to that of remark \ref{Lp_yLp_x} we have that
\begin{equation}
\Vert h(|\psi(t)|^2)\psi(t)\Vert_{H^m(\Omega_\y)}\leq \phi(T,h,\|\psi\|_{L^\infty((0,T);H^{m-1}\Omega_\y)})(1+\|\psi(t)\|_{H^m(\Omega_\y)}),\label{hpsiHm}
\end{equation}
and also,
\begin{align}
&\Vert g(v(t)) \Vert_{H^m(\Omega_\y)}\nonumber\\
&\qquad\leq \phi(T,g,\| \nabla\uu \|_{L^1(0,T;L^\infty(\Omega))},\|\rho\|_{L^\infty(0,T;H^{m-1}(\Omega))},\|\mathbf{B}\|_{L^\infty(0,T;H^{m-1})})\nonumber\\
&\qquad\qquad\qquad\qquad\qquad\qquad\qquad\cdot(1+\| \rho\|_{H^m(\Omega)}).\label{gvHm}
\end{align}

Indeed, it suffices to realize that there are universal polynomials $Q_m^k$ of degree $k$, $k=1,...,m$, such that for any two smooth enough functions $f$ and $\zeta$ we have
\[
|\nabla_\y f(\zeta(\y))|\leq \sum_{k=1}^m |f^{(k)}(\zeta)| Q_m^k(|\nabla_\y \zeta|,...,|\nabla_\y^{m-k+1}\zeta|),
\]
where $Q_m^1(z_1,...,z_m)=z_m$ and $Q_m^2$ is of the form $Q_m^2(z_1,...,z_{m-1})=C_m z_1 z_{m-1} + P_m^2(z_1,...,z_{m-2})$ for some constant $C_m$ and another universal polynomial $P_m^2$ of degree $2$. Then, if $m\geq 3$, so that $H^{m}\hookrightarrow W^{1,\infty}$, and all the derivatives of $f$ are bounded then, by the Gagliardo-Nirenberg inequality we have
\[
\| \nabla_\y^{m}f(\zeta)\|_{L^2(\Omega_\y)}\leq \phi(f,..., f^{(m)},\| \zeta\|_{H^{m-1}})(1+\|\zeta\|_{H^m(\Omega_\y)}).
\]

Putting \eqref{estschr}, \eqref{hpsiHm} and \eqref{gvHm} together we obtain the result using the induction hypothesis, lemma \ref{higheu} and Gronwall's inequality.
\end{proof}
	
\section{Local solutions}\label{S3}

Let us state the desired local existence result.
\begin{proposition} \label{localexist}
	Let $m \geq 3$ be an integer and assume, in addition to \eqref{assumptionvisc}, \eqref{assumptionp} and \eqref{assumptiongh}, that $\rho_0$, $\uu_0$, $\HH_0$ and $\psi_0$ satisfy
	\begin{equation}
	\begin{cases}
	\rho_0 &\in H^{m}(\Omega), \nonumber \\
	\uu_0 &\in H^m(\Omega), \\
	\HH_0 &\in H^m(\Omega),\\
	\psi_0 &\in H^m(\Omega_\y). \nonumber \nonumber
	\end{cases}
	\end{equation}
	Moreover, assume that there are constants $0< m_0 < M_0$ such that
	\begin{equation}
	m_0 \leq \rho_0(\x) \leq M_0
	\end{equation}
	for any $\x \in \Omega$.
	
	Then there exists $T^* = T^*(m_0,\Vert \rho_0 \Vert_{H^m}, \Vert \uu_0 \Vert_{H^m}, \|\HH_0\|_{H^m} \Vert \psi_0 \Vert_{H_\y^m})>0$ and a unique strong solution to the system \eqref{E2rho}--\eqref{E2psi} $(\rho, \uu, \psi)$ such that
	\[
	\begin{cases}
	\rho &\in C([0,T^*]; H^m(\Omega))\cap C^1([0,T^*];H^{m-1}(\Omega)), \\
	\uu  &\in L^2(0,T^*;H^{m+1}(\Omega))\cap H^1(0,T^*;H^{m-1}(\Omega)), \\
	\HH  &\in L^2(0,T^*;H^{m+1}(\Omega))\cap H^1(0,T^*;H^{m-1}(\Omega)), \\
	\psi &\in C([0,T^*];H^m(\Omega))\cap C^1( [0,T^*];H^{m-2}(\Omega)). 
	\end{cases}
	\]
\end{proposition} 

We dedicate this section to the proof of this proposition. To this end, let us introduce the space
	\begin{align*}
	E_T &= L^2(0,T;H^{m+1}(\Omega;\mathbb{R}^2)) \cap H^1(0,T; H^{m-1}(\Omega;\mathbb{R}^2)), 
	\end{align*}
	and let us define the operator
	\[K : E_T \rightarrow E_T\]
where $\vv = K \uu$ is the solution of the following linear parabolic problem on the $2$-torus $\Omega$
\begin{equation}
	\begin{cases} \label{eqKu}
	\rho \vv_t + (L_\rho \vv) + \rho (\uu \cdot \nabla) \vv= \HH\cdot\nabla \HH - \tfrac{1}{2}\nabla|\HH|^2- \nabla p\\
	\qquad\qquad\qquad\qquad\qquad\qquad\qquad+  \alpha \nabla (g'(1/\rho)h(|\psi \circ \Y|^2) \J_\y/\rho),\\
	\vv(\x,0) = \uu_0(\x).
	\end{cases}
	\end{equation}
	with $\rho$, $\HH$, $\Y$ and $\psi$ being the corresponding density, magnetic field, Lagrangian transformation and wave function associated to $\uu$, according to the considerations from the previous section. Then we have the following.
	
	\begin{lemma} \label{boundness}
		There exists a constant $N$ depending on $m_0$, $\Vert \rho_0 \Vert_{H^m(\Omega)}$, $\|\uu_0\|_{H^m(\Omega)}$, $\|\HH_0\|_{H^m(\Omega)}$ and $\Vert \psi_0 \Vert_{H^m(\Omega_\y)}$ with the following property: given $r > N$, there is a positive $T^* = T^*(r, m_0, \Vert \rho_0 \Vert_{H^m}, \|\uu_0\|_{H^m}, \Vert \psi_0 \Vert_{H_\y^m})$ for which
		\[
		K : r B_{E_{T^*}} \rightarrow r B_{E_{T^*}},
		\]
		where $rB_{E_T}$ stands for the closed ball in $E_T$ with radius is $r$.
		
		In other words, $r B_{E_{T^*}}$ is invariant under $K$.
	\end{lemma}
	\begin{proof}
	\textit{Step \#1.}
	
		Given $\uu \in E_T$, lemma \ref{highparabol} asserts that
		\begin{align}
		&\Vert K \uu \Vert_{E_T}^2 \leq \phi\Big(T, \Big(\operatornamewithlimits{Min}_{\x \in \Omega, 0\leq t \leq T} \rho(\x,t)\Big)^{-1} , \Vert \rho \Vert_{ W^{1,\infty}(0,T; L^2(\Omega))\cap C([0,T];H^m)}\Big) \nonumber\\ 
		&\cdot \big(\Vert \rho (\uu \cdot \nabla) K\uu  \Vert_{L^2(0,T;H^{m-1})}^2 +\|\nabla\HH\cdot\HH\|_{L^2(0,T;H^{m-1})} +\Vert \nabla p  \Vert_{L^2(0,T;H^{m-1})}^2   \nonumber\\
		&\qquad\qquad+ \big\Vert \alpha \nabla (g'(1/\rho)h(|\psi \circ \Y|^2) \J_\y/\rho) \big\Vert_{L^2(0,T;H^{m-1})}^2 + \Vert \uu_0 \Vert_{H^m}^2\big).\label{estK}
		\end{align}
	
Note that from \eqref{J/rho=const} we have that 
\begin{equation}
\frac{\J_\y}{\rho}(t,\x) = \rho_0(\y(t,\x)),\label{Jy/rho(x)}
\end{equation} 
Then, using the fact that $H^{m-1}(\Omega)$ is an algebra, as well as the embedding $E_T\hookrightarrow C(0,T;C^1(\Omega))$, together with lemmas \ref{higheu}, \ref{lemmaH(u)}, \ref{existschr} and Corollary~\ref{equivnorms} we have
\begin{align}
	&\Vert K \uu \Vert_{E_T}^2 \leq \phi\big(T, (m_0)^{-1}, \Vert \rho_0 \Vert_{H^m}, \Vert \uu_0 \Vert_{H^m},\|\HH_0\|_{H^m},\Vert \psi_0 \Vert_{H_\y^m},  \Vert \uu \Vert_{E_T}\big) \nonumber\\
	&\qquad\qquad\qquad\qquad\qquad\qquad\cdot \big(1 + \Vert K \uu \Vert_{L^2(0,T;H^m)}^2\big) . \label{estKf}
	\end{align}

\textit{	Step \#2.}

	Regarding the term $\Vert K \uu \Vert_{L^2(0,T;H^m)}$ appearing in the estimate above, we apply the Gagliardo--Nirenberg inequality to obtain for every $\varepsilon > 0$
	\begin{align}
	\Vert K \uu \Vert_{L^2(0,T;H^m)}^2 &\leq \varepsilon^2 	\Vert K \uu \Vert_{L^2(0,T;H^{m+1})}^2 + C_\varepsilon 	\Vert K \uu \Vert_{L^2(0,T;L^2)}^2 \nonumber\\
	&\leq \varepsilon^2 	\Vert K \uu \Vert_{E_T}^2 + C_\varepsilon 	\Vert K \uu \Vert_{L^2(0,T;L^2)}^2.
	\end{align}
	Choosing appropriately $\varepsilon = \varepsilon(\phi(T, \ldots))$, 
	\eqref{estKf} is somewhat simplified to
	\begin{align}
	 &\Vert K \uu \Vert_{E_T}^2 \leq \phi\big(T, (m_0)^{-1}, \Vert \rho_0 \Vert_{H^m}, \Vert \uu_0 \Vert_{H^m},\|\HH_0\|_{H^m},\Vert \psi_0 \Vert_{H_\y^m},  \Vert \uu \Vert_{E_T}\big) \nonumber\\
	 &\qquad\qquad\qquad\qquad\qquad\qquad\cdot \big(1 + \Vert K \uu \Vert_{L^2(0,T;L^2)}^2\big) . \label{estKf1}
	 \end{align}
	To get rid of this final term in the right-side, multiply \eqref{eqKu} by $K\uu$ to obtain
	\begin{align}
	&\frac{1}{2}\frac{d}{dt} \int_{\Omega} \rho(\x,t) \big\vert K\uu(\x,t) \big\vert^2 d\x + \mu \cdot \int_{\Omega}  \big|\nabla K \uu(\x,t) \big|^2 d\x \nonumber \\
	&\qquad + \int_{\Omega} \big( \mu + \lambda\big(\rho(\x,t)\big) \big(\Div K \uu(\x,t)\big)^2  d\x  \nonumber \\ 
	&= \int_\Omega (\HH\cdot\nabla \HH - \tfrac{1}{2}\nabla|\HH|^2)\cdot K\uu d\x + \int_{\Omega} p(\x,t) \cdot \Div K\uu(\x,t) d\x  \nonumber\\
	&\qquad- \alpha \int_\Omega g'(1/\rho) h(|\psi \circ \Y|^2) \J_\y/\rho \cdot \Div K\uu(\x,t) d\x;
	\end{align}
	where we have used the continuity equation \eqref{continuity} in order to write
	\[
	\frac{1}{2}\frac{d}{dt} (\rho |K\uu|^2) = \rho (K\uu)_t\cdot K\uu + \rho(\uu\cdot\nabla) K\uu \cdot K\uu - \frac{1}{2}\Div(\rho \uu |K\uu|^2).
	\]
	
	Thus by virtue of \eqref{J/rho} and \eqref{assumptiongh}, using Young's inequality we conclude that
	\begin{align}
	\frac{1}{2}\frac{d}{dt} &\int_{\Omega} \rho(\x,t) \big\vert K\uu(\x,t) \big\vert^2 d\x + \mu \int_{\Omega} \big|\nabla K \uu (\x,t)\big|d\x \nonumber\\
	&\leq C_\mu (1+\|K\uu(t)\|_{L^2(\Omega)}+\|p(\rho(t))\|_{L^2(\Omega)} + \|\HH\|_{L^4(\Omega)}^2\|\nabla\HH\|_{L^4(\Omega)}^2) \nonumber,
	\end{align}
	and again, using the estimates from the previous section, Gronwall's inequality yields
	\begin{align}
	\operatornamewithlimits{Max}_{0\leq t \leq T}&\int_\Omega \big|K\uu (\x,t)\big|^2 d\x + \int_{0}^{T}\big|\nabla \uu(\x,t)\big|^2 d\x dt \nonumber\\
	&\qquad\qquad\leq \phi\big(T, (m_0)^{-1}, \Vert \rho_0 \Vert_{H^m}, \Vert \uu_0 \Vert_{H^m},\|\HH_0\|_{H^m},\Vert \psi_0 \Vert_{H_\y^m},  \Vert \uu \Vert_{E_T}\big). \nonumber
	\end{align}
	
	Inserting this estimate in \eqref{estKf1} we obtain,
	\begin{align}
	\Vert K \uu \Vert_{E_T}^2 \leq \phi\big(T, (m_0)^{-1}, \Vert \rho_0 \Vert_{H^m}, \Vert \uu_0 \Vert_{H^m},\|\HH_0\|_{H^m},\Vert \psi_0 \Vert_{H_\y^m},  \Vert \uu \Vert_{E_T}\big).\label{estKf2}
	\end{align}

\textit{Step \#3.} (Conclusion)
	
	Let us recall that $\phi$ appearing on \eqref{estKf2} is a continuous and non-decreasing function of its arguments. Moreover, we note that it can be most naturally chosen as to satisfy
	\begin{align*}
	\phi(T=0, (m_0)^{-1}, \Vert \psi_0 \Vert_{H^m}, \Vert \rho_0 \Vert_{H^m}, \Vert \uu_0 \Vert_{H^m}, \eta) &\leq \text{constant independent of } \eta, \\
	&= N \\
	&= N(\Vert \rho_0 \Vert_{H^m}, m_0, \Vert \psi_0 \Vert_{H^m}).
	\end{align*}
	In that way, for $r > N$, by continuity there exists a $T^*>0$ such that 
	\[\phi(T^*, (m_0)^{-1}, \Vert \psi_0 \Vert_{H^m}, \Vert \rho_0 \Vert_{H^m}, \Vert \uu_0 \Vert_{H^m}, r) < r^2.
	\]
	As a consequence, \eqref{estKf2} says that $K : r B_{E_{T^*}} \rightarrow r B_{E_{T^*}}$.
	
	\end{proof}
	
Let us now introduce the space 
\[
F_T = C([0,T];L^2(\Omega;\mathbb{R}^2)) \cap L^2(0,T; \dot{H}^1(\Omega;\mathbb{R}^2)).
\]
Then, we can prove that the constant $T^*$ from lemma \ref{boundness} above can be chosen so that the application $K$, restricted to $rB_{E_{T^*}}$, is contractive in the norm of $F_T$.

\begin{lemma} \label{contraction}
	Given any $0 < \kappa < 1$, one can choose a $r, T > 0 $, depending only on $r, \Vert \rho_0 \Vert_{H^m}, m_0, \Vert \psi_0 \Vert_{H^m}$, for which, not only
	$$K : r B_{E_T} \rightarrow r B_{E_T},$$
	but also for every $\uu'$ and $\uu''$ in $r \cdot B_{E_T}$, one has that
	\begin{equation}
	\big\Vert K\uu' - K\uu'' \big\Vert_{F_T} \leq \kappa \cdot 	\big\Vert \uu' - \uu'' \big\Vert_{F_T}.
	\end{equation}
\end{lemma}

\begin{proof}
\textit{	Step \#1.}

	Let $\vv' = K\uu'$ and let us denote by $\rho'$, $\HH'$, $\psi'$ and $\Y'=(t,\y'(t,\cdot))$ the density, magnetic field, wave function and Lagrangian transformation associated to $\uu'$. Correspondingly, we write $\vv''=K\uu''$ and denote by $\rho''$, $\HH''$, $\psi''$ and $\Y''=(t,\y''(t,\cdot))$ the density, magnetic field, wave function and Lagrangian transformation associated to $\uu''$. Let us further denote $\tilde{\vv} = \vv' - \vv''$, $\tilde{\rho}= \rho' - \rho''$, $\tilde{\HH} = \HH'-\HH''$, $\tilde{\psi} = \psi' - \psi''$, and $\tilde{\Y} = \Y'-\Y''$.

	Then, with some effort, we can show that $\tilde{\vv}$ satisfies the following equation
	\begin{equation}
	\begin{aligned}
	\rho' &\tilde{\vv}_t  + \rho'(\uu'\cdot \nabla) \tilde{\vv} - L_{\rho'} \tilde{\vv} = \\
	 &- \tilde{\rho} {\vv}_t''	- (L_{\rho'}-L_{\rho''})\vv''-\rho'(\tilde{\vv}\cdot\nabla)\vv'' - \tilde{\rho} (\uu' \cdot \nabla ) \vv''\\
      &+ \tilde{\HH}\cdot \nabla\HH' + \HH''\cdot\nabla \tilde{\HH} - \nabla\HH'\cdot\tilde{\HH}-\nabla\tilde{\HH}\cdot\HH''\\
      &- \nabla (p(\rho') - p(\rho'')) + \nabla \big(\big( g'(1/\rho') - g'(1/\rho'') \big) h(|\psi' \circ \Y'|^2) \J_{\y'}/\rho' \big)\\
      &+\nabla \big( g'(1/\rho'') \big(h(|\psi' \circ \Y'|^2) - h(|\psi' \circ \Y''|^2) \big) J_{\y'}/\rho'\big) \\
      &+\nabla \big( g'(1/\rho'') \big(h(|\psi' \circ \Y''|^2) - h(|\psi'' \circ \Y''|^2) \big) J_{\y'}/\rho'\big)\\
      &+\nabla \big( g'(1/\rho'')  h(|\psi'' \circ \Y''|^2) \big(J_{\y'}/\rho' - J_{\y''}/\rho'' \big).
	\end{aligned}
	\end{equation}
	
	Multiplying the equation above by $\tilde{\vv}$ and integrating by parts we arrive at
	\begin{equation}	
	\begin{aligned}
	 \frac{d}{dt} \Big\{& \frac{1}{2}\int_{\Omega} \rho' |\tilde{\vv}|^2 d\x \Big\}   + \int_\Omega (\mu |\nabla \tilde{\vv}|^2+(\lambda(\rho')+\mu)(\Div\tilde{\vv})^2) d\x   \\
	=&- \int_{\Omega} \tilde{\rho} {\vv}_t'' \cdot \tilde{\vv} d\x	- \int_{\Omega} (\lambda(\rho')-\lambda(\rho'')) \Div \vv'' \Div \tilde{\vv} d\x\\
	&-\int_{\Omega} \rho'[(\tilde{\vv}\cdot\nabla)\vv''] \cdot\tilde{\vv} d\x - \int_{\Omega} \tilde{\rho} [(\uu' \cdot \nabla ) \vv''] \cdot \tilde{\vv} d\x \\
	&+\int_\Omega (\tilde{\HH}\cdot \nabla\HH' + \HH''\cdot\nabla \tilde{\HH} - \nabla\HH'\cdot\tilde{\HH}-\nabla\tilde{\HH}\cdot\HH'')\cdot\tilde{\vv} d\x \\
	&+ \int_\Omega (p(\rho') - p(\rho'')) \Div \tilde{\vv} d\x  	\\
	&- \int_\Omega \big( g'(1/\rho') - g'(1/\rho'') \big) h(|\psi' \circ \Y'|^2) \frac{\J_{\y'}}{\rho'}\Div \tilde{\vv} d\x\\
	&- \int_{\Omega}  g'(1/\rho'') \big(h(|\psi' \circ \Y'|^2) - h(|\psi' \circ \Y''|^2) \big) \frac{J_{\y'}}{\rho'} \Div \tilde{\vv} d\x \\
	&- \int_\Omega g'(1/\rho'') \big(h(|\psi' \circ \Y''|^2) - h(|\psi'' \circ \Y''|^2) \big) \frac{J_{\y'}}{\rho'} \Div \tilde{\vv} d\x \\
	&- \int_{\Omega} g'(1/\rho'')  h(|\psi'' \circ \Y''|^2) \left(\frac{J_{\y'}}{\rho'} - \frac{J_{\y''}}{\rho''}\right) \Div \tilde{\vv} d\x.\label{coisagigantesca}
	\end{aligned}
	\end{equation}
	
	Note that, through some applications of Young's inequality with $\varepsilon$ and in light of lemma \ref{boundness}, we can arrive at the desired conclusion, provided that we bound appropriately $\tilde{\rho}$, $\tilde{H}$, $\tilde{\psi}$, $\tilde{Y}$ and $\frac{J_{\y'}}{\rho'} - \frac{J_{\y''}}{\rho''}$.
	
\textit{Step \#2.}\ (Analysis of $\tilde{\rho}$). 

A straightforward calculation shows that $\tilde{\rho}$ satisfies the following equation
	\begin{equation*}
		\begin{dcases*}
	\frac{\partial \tilde{\rho}}{\partial t} + \Div (\tilde{\rho} \uu' + \rho'' (\uu' - \uu'')) = 0\\
	\tilde{\rho}(0,\x) = 0.
	\end{dcases*}
	\end{equation*}
	So, multiplying the equation by $\tilde{\rho}$ yields
	\begin{align*}
	 &\frac{1}{2}\frac{d}{dt} \int_\Omega \tilde{\rho}^2 d\x\\
	 &\qquad= - \frac{1}{2} \int_\Omega (\Div \uu' ) \tilde{\rho}^2 d\x - \int_{\Omega} \tilde{\rho}\nabla \rho'' \cdot (\uu' - \uu'')  d\x - \int_{\Omega} \tilde{\rho}\rho'' \Div (\uu' - \uu'')  d\x \\
	 &\qquad\leq \frac{1}{2}  \Vert \Div \uu' (t) \Vert_\infty  \cdot \Vert \tilde{\rho} (t) \Vert_2^2 + \Vert \nabla \rho'' (t) \Vert_\infty \cdot \Vert \tilde{\rho}(t) \Vert_2 \cdot \Vert \uu'(t) - \uu'' (t) \Vert_2 \\
	 &\qquad\qquad + \Vert  \rho'' (t) \Vert_\infty \cdot \Vert \tilde{\rho}(t) \Vert_2 \cdot \Vert \Div \uu'(t) - \Div \uu'' (t) \Vert_2;
	\end{align*}
	hence, by Gronwall's inequality and the fact that $\uu \in C([0,T^*];H^3)$ and $\rho'' \in C([0,T^*];H^3)$,
	\begin{equation}
	\int_\Omega \tilde{\rho}^2 (\x,t) d\x \leq C  \Vert \uu' - \uu'' \Vert_{L^2(0,T^*;H^1(\Omega))}^2\label{tilderho}
	\end{equation} 
	for all $0\leq t \leq T^*$ and some constant $C$ depending on $T^*$ and $r$.
	
\textit{Step \#3.}\ (Analysis of $\tilde{\HH}$). 

Similarly, we see that $\tilde{\HH}$ is the solution of the following equation
\[
\begin{cases}
\tilde{\HH}_t - \nu\Delta \tilde{\HH} = -(\uu'-\uu'')\cdot\nabla\HH'' - \uu'\cdot\nabla\tilde{\HH} - \tilde{\HH}\cdot\nabla\uu'-\HH''\cdot\nabla(\uu'-\uu'') \\
\qquad\qquad\qquad\qquad+ \tilde{\HH}\Div\uu' + \HH''\Div(\uu'-\uu''),\\
\tilde{\HH}(0,\x)=0.
\end{cases}
\]
	
Multiplying by $\tilde{\HH}$ and integrating, upon integrating by parts and using Young's inequality with $\varepsilon$ we get
\begin{align*}
&\frac{1}{2}\frac{d}{dt}\int_\Omega |\tilde{\HH}|^2d\x + \nu\int_\Omega |\nabla\tilde{\HH}|^2d\x\\
& \leq \varepsilon \|\uu'\|_\infty^2 \int_\Omega |\nabla\tilde{\HH}|^2d\x + C_\varepsilon \Big(\|\nabla \HH'' \|_\infty^2\|\uu'-\uu''\|_2^2+ \|\HH''\|_\infty^2\|\nabla(\uu'-\uu'')\|_2^2 \\
&\qquad\qquad\qquad\qquad\qquad\qquad\qquad\qquad\qquad(1+\|\nabla\uu'\|_\infty)\|\tilde{\HH}\|_2^2\Big),
\end{align*}
and choosing $\varepsilon>0$ small enough, by Gronwall's inequality and the fact that $\uu \in C([0,T^*];H^3)$ and $\HH'' \in C([0,T^*];H^3)$,
\begin{equation}
\int_\Omega \tilde{\HH}^2 (\x,t) d\x+\int_0^t\int_\Omega|\nabla\tilde{\HH}|^2d\x  \leq C  \Vert \uu' - \uu'' \Vert_{L^2(0,T^*;H^1(\Omega))}^2
\end{equation} 
for all $0\leq t \leq T^*$ and some constant $C$ depending on $T^*$ and $r$.
	
\textit{Step \#4.}\ (Analysis of $\tilde{\y}$).

	From the definition of the Lagrangian coordinate we see that $\y'$ satisfies the equation
	\begin{equation*}
	\y'_t(t,\x) + \uu'(t,\x)\cdot \nabla_\x \y'(t,\x)=0,\qquad
	\y'(0,\x)=\x.
	\end{equation*}
    By symmetry, $\y''$	 satisfies an analogous equation with $\uu''$ in place of $\uu'$. Then, the difference $\tilde{\y}=\y'-\y''$ satisfies
    \begin{equation*}
    \begin{cases}
	\tilde{\y}_t =- (\uu'-\uu'')\cdot \nabla_\x \y' - \uu''\cdot\nabla\tilde{\y},\\
	\tilde{\y}(0,\x)=0.
	\end{cases}
	\end{equation*}
	
	Multiplying by $\tilde{\y}$ and integrating by parts we obtain
	\begin{equation*}
	\frac{d}{dt}\int_\Omega |\tilde{\y}|^2d\x\leq C(\|\nabla\y'\|_\infty^2\|\uu'-\uu''\|_2 + \|\Div\uu''\|_\infty^2\|\tilde{\y}\|_2^2),
    \end{equation*}		
and then, Gronwall's inequality implies
\begin{equation}
 \int_\Omega |\tilde{\y}|^2d\x \leq C\|\uu'-\uu''\|_{L^2(0,T^*;H^1(\Omega))}^2.
\end{equation}
	
In the interest of estimating the wave functions, we have to analyse also the inverse Lagrangian transforms $\x'(t,\y)$ and $\x''(t,\y)$ corresponding to the velocity fields $\uu'$ and $\uu''$, respectively. In that direction we have that the difference $\tilde{\x}:=\x'-\x''$ solves the following equation
\begin{equation*}
 \begin{cases}
  \tilde{\x}_t(t,\y) = \uu'(t,\x'(t,\y))-\uu''(t,\x''(t,\y)), \\
  \tilde{\x}(0,\y)=0.  
 \end{cases}
\end{equation*}

Writing
\begin{align*}
  \uu'(t,\x'(t,\y))-\uu''(t,\x''(t,\y)) = &\>\uu'(t,\x'(t,\y))-\uu''(t,\x'(t,\y)) \\  &+ \uu''(t,\x'(t,\y))-\uu''(t,\x''(t,\y)),
\end{align*} and multiplying the equation above by $\tilde{\x}$ we have
\begin{align*}
&\frac{d}{dt}\int_{\Omega_y} |\tilde{\x}|^2d\y \\
&\leq C\left(\int_{\Omega_y} |\uu'(t,\x'(t,\y))-\uu''(t,\x'(t,\y))|^2d\y + (1+\|\nabla_x\uu'\|_\infty) \int_{\Omega_\y}|\tilde{\x}|^2d\y\right)
\end{align*}

And, therefore, using Gronwall's inequality together with Corollary~\ref{equivnorms} we obtain
\begin{equation}
\int_{\Omega_\y}|\tilde{\x}|^2d\y\leq C\|\uu'-\uu''\|_{L^2(0,T^*;H^1(\Omega))}^2.\label{tildex}
\end{equation}

	\textit{Step \#5.} (Analysis of $\tilde{\psi}$).
	
	 Now, $\tilde{\psi}(t,\y) = \psi'(t,\y) - \psi''(t,\y)$ solves the Schrödinger equation
	\begin{equation*}
	\begin{dcases}
	i \tilde{\psi}_t + \Delta_\y \tilde{\psi} = |\psi'|^2\psi' - |\psi''|^2\psi'' + \big(g(v') - g(v'')\big) h'(|\psi'|^2)\psi' 
	\\ \qquad\qquad\qquad\qquad+ g(v'')\big(h'(|\psi'|^2) \psi' - h'(|\psi''|^2) \psi''\big), \\
	\tilde{\psi}(\y,0) = 0,
	\end{dcases}
	\end{equation*}
 where, $v'$ and $v''$ are given by
 \[
 v'(t,\y)=\frac{1}{\rho'(t,\x'(t,\y))},\qquad\qquad v''(t,\y)=\frac{1}{\rho''(t,\x''(t,\y))}.
 \]
	
	Thus, multiplying the equation by $\tilde{\psi}$ and taking imaginary parts, after some manipulation we obtain
	\begin{align}
	\frac{d}{dt} \int_{\Omega_\y} |\tilde{\psi}|^2 d\y &\leq C \cdot \Big( \int_{\Omega_\y} |\tilde{\psi}|^2 d\y + \int_{\Omega_\y} |v'(t,\y)-v''(t,\y)|^2d\y \Big),\label{tildepsi'}
	\end{align}
	
Regarding the last term on the right hand side, using the uniform bounds available on the densities, we have that
\begin{align*}
&\int_{\Omega_\y} |v'(t,\y)-v''(t,\y)|^2d\y \\
&\qquad\leq C \int_{\Omega_\y} |\rho'(t,\x'(t,\y))-\rho''(t,\x''(t,\y))|^2d\y\\
&\qquad\leq C\Bigg( \int_{\Omega_\y} |\rho'(t,\x'(t,\y))-\rho''(t,\x'(t,\y))|^2d\y \\
&\qquad\qquad\qquad\qquad\qquad\qquad+ \|\nabla\rho''\|_\infty\int_{\Omega_\y} |\x'(t,\y)-\x''(t,\y)|^2d\y\Bigg)\\
&\qquad\leq C\|\uu'-\uu''\|_{L^2(0,T^*;H^1(\Omega))},
\end{align*}
where we have used \eqref{tildex}, \eqref{tilderho} and Corollary~\ref{equivnorms}.

Coming back to \eqref{tildepsi'} and using Gronwall's inequality we conclude
\begin{equation}
\int_{\Omega_\y}|\tilde{\psi}|^2d\y \leq C\|\uu'-\uu''\|_{L^2(0,T^*;H^1(\Omega))}^2,
\end{equation}

\textit{Step \#6.} (Analysis of $\frac{J_{\y'}}{\rho'} - \frac{J_{\y''}}{\rho''}$ and conclusion)

Finally, recalling \eqref{Jy/rho(x)} we note that $\frac{J_{\y'}}{\rho'}(t,\y) - \frac{J_{\y''}}{\rho''}(t,\y) = \rho_0(\x'(t,\y))-\rho_0(\x''(t,\y))$; and reasoning similarly as before we have
\begin{equation}
\int_\Omega \left| \frac{J_{\y'}}{\rho'} - \frac{J_{\y''}}{\rho''} \right|^2(t,\x) d\x \leq \|\uu'-\uu''\|_{L^2(0,T^*;H^1(\Omega))}^2.
\end{equation}

Analysing each one of the terms in \eqref{coisagigantesca} at the light of the estimates above, we conclude that its right-hand side may be bounded by
	$$\varepsilon \cdot\int_{\Omega} \Big( |\tilde{\vv}|^2 + |\nabla \tilde{\vv}|^2 \Big) d\x  +  C_\varepsilon \cdot  \Vert \uu' - \uu'' \Vert_{L^2(0,T;H^1)}, $$
	where $C_\varepsilon = C(\varepsilon, r, T).$ Thus, integrating \eqref{coisagigantesca} from $t' = 0$ to $t' = t <T$  and taking a sufficiently small $\varepsilon = \varepsilon(r, T^*,\mu)>0$,
	\begin{align*}
	&\int_\Omega \rho'(\x,t) |\tilde{\vv}(\x,t)|^2 d\x + \int_0^t \int_{\Omega} |\nabla \tilde{\vv} (\x,t')|^2 d\x \leq C \cdot T \cdot \Vert \uu - \uu \Vert_{F_t}.
	\end{align*}
  At last, diminishing $T^* = T^*(r,\varepsilon)$ accordingly, the desired conclusion follows.
\end{proof}

\begin{proof} [Proof of proposition \ref{localexist}]
	For, say, $\kappa = 1/2$, let $r$ and $T^*$ be as in the previous lemma. Then, for $\uu^{(0)} \equiv 0$, define recursively $$\uu^{(n)} = K \uu^{(n-1)}.$$
	Since $K$ preserves the ball of radius $r$ in $E_T$, $\uu^{(n)}$ defines a bounded sequence in $E_T$. Moreover, from lemma \ref{contraction} we know that
	\[\uu^{(n)} \rightarrow \uu \text{ strongly in } F_T = C([0,T];L^2) \cap L^2(0,T;H^1),
	\]
where $\uu$ is a fixed point of $K$.
	
By an elementary weak convergence argument, 
\[\uu \in E_T = L^2(0,T;H^{m+1}) \cap H^1(0,T;H^{m-1}),
\]
and, in fact, 
\[\uu^{(n)} \rightharpoonup \uu \text{ weakly in } E_T.
\] 
In particular, by interpolation, $$\uu^{(n)} \rightarrow \uu \text{  strongly in } L^2(0,T;H^{m}).$$
	
	Because of this convergence, it is not hard to see that the associated densities $\rho^{(n)}$, magnetic fields $\HH^{(n)}$ and wave functions $\psi^{(n)}$ also converge 
	$$\begin{dcases}
\rho^{(n)} \rightarrow \rho \text{ strongly in } C([0,T];H^{m-1}(\Omega)), \\
\HH^{(n)} \rightarrow \HH \text{ strongly in } C([0,T];H^{m-1}(\Omega)),\\
\psi^{(n)} \rightarrow \psi \text{ strongly in } C([0,T];H^{m-1}(\Omega_\y))
	\end{dcases}$$
	for some $\rho \in C([0,T];H^{m}(\Omega))$ (which is positive), $\HH\in L^2(0,T;H^{m+1}(\Omega))\cap H^1(0,T;H^{m-1}(\Omega))$  and $\psi \in C([0,T];H^{m}(\Omega_\y))$; wherein the limit $[\rho, \uu, \HH, \psi]$ is a strong solution to the system \eqref{E2rho}--\eqref{E2psi} with the desired regularity.
	
	This proves the existence part of the proposition. The uniqueness assertion follows trivially from the lemma \ref{contraction}.
\end{proof}

\section{Low order a priori estimates}\label{S4}

In this section we deduce the first set of a priori estimates. These low order a priori estimates are the starting point for the proof of the existence of global smooth solutions. In particular, we prove that solutions are away from vacuum, so that the Lagrangian transformation makes sense.

We point out that the $H^m$ regularity, with $m\geq 3$, from Theorem~\ref{principalthm} was necessary for the fixed point argument in the proof of the existence of local solutions. Indeed, this is necessary in the calculations that ensure the application $K$ to be a contraction, where we  required the wave function and the specific volume to be Lipschitz continuous. This is due to the coupling terms, where we have a composition of the wave function with the Lagrangian transformation.

The a priori estimates from this section are based on the analogues by Yu Mei in \cite{Mei,Mei'} on the periodic solutions to the 2-dimensional MHD equations, which, in turn, are inspired by the work of Va\u \i gant and Kazhikhov \cite{VK}, by the work of Perepelitsa in \cite{P} and by the work of Jiu, Wang and Xin \cite{JWX} on the 2-dimensional periodic Navier-Stokes equations.

The crucial a priori estimates are the uniform bounds from above and away from vacuum for the density. As delineated in the introduction, these estimates are achieved by analyzing a transport equation involving a function of the density, namely, equation \eqref{transportrho} below. The desired bounds will follow from the key estimates on the commutators $[u_i,R_i R_j](\rho u_j)$ and $[H_i,R_i R_j](H_j)$ of Riesz transforms and the operators of multiplication by $u_i$ and $H_i$. Such estimates ultimately depend on a careful analysis of the effective viscous flux $\F$ (given by \eqref{EVF}), which, in turn, relies on some $L^p$ bounds on the density.

Note that the coupling term from the momentum equation \eqref{E2u} behaves as a pressure and, accordingly, we incorporate it in the definition of the effective viscous flux. This is very advantageous for us as the elliptic estimates from lemmas \ref{elliptic1} and \ref{elliptic2} turn out to be identical to the analogues in \cite{Mei}. However, this additional term has to be accounted for in the subsequent estimates which poses a delicate interplay between the estimates from the MHD equations and from the NLS equation. At this point, the regularity of the Lagrangian transformation comes into play and is crucial in order to close the estimates. This is most evident in the proof of Lemma \ref{lemmanablarho} where we see explicitly how the regularity of the density depends on the integrability of the deformation gradient.

In this section, we follow \cite{Mei,Mei'}, closely making emphasis on the modifications that have to be made in order to include the short wave-long wave interactions.

\subsection{Energy estimates and uniform bounds for the density}

Let us begin with the basic energy estimate, which follows directly from the energy identity \eqref{difE}.
\begin{lemma}\label{apriorienergy}
Let
\begin{align*}
&E(t):=\int_\Omega \left( \rho\left(\frac{1}{2}|\uu|^2 + e \right) + \frac{1}{2}|\HH|^2 \right)d\x \\
&\qquad\qquad+ \int_{\Omega_\y}\left( \frac{1}{2}|\nabla_\y \psi|^2 + \frac{1}{4}|\psi|^4 + \alpha g(v)h(|\psi|^2) \right) d\y.
\end{align*}

Then,
\begin{equation}
E(t) + \int_0^t \int_\Omega\left( \mu |\nabla \uu|^2 + (\lambda(\rho)+\mu)(\Div\uu)^2 + \nu|\nabla\HH|^2 \right)d\x ds = E(0).\label{energy}
\end{equation}
\end{lemma}

In light of estimate \eqref{energy}, part (ii) of lemma \ref{H1-H2(u)} yields the following higher order estimate on the wave function $\psi$.

\begin{lemma}\label{aprioripsi}
\begin{equation}
\int_{\Omega_\y} \big( |\psi_t|^2+|\Delta_\y \psi|^2 \big) d\y\leq C.\label{estpsit}
\end{equation}
In particular, for a.e. $t\in [0,T]$
\begin{equation}
\max_{\y\in\Omega_\y}|\psi(t)|\leq C.\label{uniformpsi}
\end{equation}
\end{lemma}

Moving on, through some standard manipulation, multiplying equation \eqref{E2H} by $p|\HH|^{p-2}\HH$ and integrating we deduce the following $L^p$ estimate on the magnetic field. We omit the details.
\begin{lemma}\label{lemmaHLp}
For any $p\geq 2$, there exists a positive constant $C$ such that
\begin{equation}\label{HLp}
\sup_{0\leq t\leq T} ||\HH(t)||_{L^p(\Omega)}\leq C.
\end{equation}
\end{lemma}

Note that the momentum equation \eqref{E2u} can be written as
\begin{equation}
(\rho \uu)_t + \Div(\rho\uu \otimes\uu - \HH\otimes\HH) = \mu\nabla^\perp \omega + \nabla \F, \label{E2u'}
\end{equation}
where, $\nabla^\perp$ is the operator $\nabla^\perp:=(\partial/\partial x_2,-\partial/\partial x_1)$, $\omega:=\nabla^\perp\cdot \uu$ is the vorticity and $\F$ is the effective viscous flux given by
\begin{equation}
\F:= (2\mu+\lambda(\rho))\Div \uu - P(\rho) - \frac{1}{2}|\HH|^2+\alpha g'(1/\rho)h(|\psi\circ \Y|^2)\frac{\J_\y}{\rho}. \label{EVF}
\end{equation}

Applying the divergence operator to equation \eqref{E2u'} we see that
\begin{equation}
\Delta\left(\xi_t + \eta + \F -\int_\Omega\F(\tilde{\x},t)d\tilde{\x}\right)=0,\label{E2u''}
\end{equation}
with
\[
\int_\Omega \left( \xi_t + \eta + \F -\int_\Omega\F(\tilde{\x},t)d\tilde{\x} \right) d\x = 0,
\]
where, $\xi$ and $\eta$ are the solutions to the following elliptic problems
\begin{equation}
-\Delta \xi=\Div(\rho \uu),\qquad\int_\Omega \xi(\x,t)d\x = 0,\label{defxi}
\end{equation}
\begin{equation}
-\Delta\eta = \Div \big(\Div(\rho\uu\otimes\uu -\HH\otimes\HH) \big), \qquad\int_\Omega \eta(\x,t)d\x=0.\label{defeta}
\end{equation}

From \eqref{E2u''} we infer that
\[
\xi_t + \eta + \F -\int_\Omega\F(\tilde{\x},t)d\tilde{\x} =0.
\]
Thus, defining
\begin{equation}
\Lambda(\rho):=\int_1^\rho \frac{2\mu+\lambda(s)}{s}ds = 2\mu\log\rho + \frac{1}{\beta}(\rho^\beta -1),
\end{equation}
and using the continuity equation along with the definition of $\F$ we arrive at the following transport equation
\begin{flalign}
&(\Lambda(\rho)-\xi)_t + \uu\cdot\nabla(\Lambda(\rho)-\xi) + P + \frac{|\HH|^2}{2} \label{transportrho}&\\
&\qquad\qquad\qquad-\alpha g'(1/\rho)h(|\psi\circ\Y|^2)\frac{\J_\y}{\rho} - \eta+ u\cdot\nabla\xi +\int_\Omega\F(\tilde{\x},t)d\tilde{\x} = 0.&\nonumber
\end{flalign}

Regarding the functions $\xi$ and $\eta$, just as in \cite{Mei,VK,JWX}, we have the following elliptic estimates.
\begin{lemma}\label{elliptic1}\hfill
\begin{enumerate}
\item[(1)]$\Vert \nabla\xi\Vert_{2m}\leq Cm\Vert \rho\Vert_{\frac{2mk}{k-1}}\Vert\uu\Vert_{2mk},$ for any $k>1$ and $m\geq 1$;
\item[(2)] $\Vert \nabla\xi\Vert_{2-r}\leq C\Vert \rho\Vert_{\frac{2-r}{r}}^{1/2},$ for any $0<r<1$;
\item[(3)] $\Vert \eta\Vert_{2m}\leq Cm(\Vert \rho \Vert_{\frac{2mk}{k-1}}\Vert \uu\Vert_{4mk}^2 + \Vert \HH\Vert_{4m}^2),$ for any $k>1$ and $m\geq 1$,
\end{enumerate}
\vskip 5pt
where $C$ is a positive constant independent of $m$, $k$ and $r$.
\end{lemma}

\begin{lemma}\label{elliptic2}\hfill
\begin{enumerate}
\item[(1)]$\Vert \xi\Vert_{2m}\leq Cm^{1/2}\Vert\nabla\xi \Vert_{\frac{2m}{m+1}}\leq Cm^{1/2}   \Vert \rho\Vert_{m}^{1/2},$ for any $m\geq 2$;
\item[(2)] $\Vert \uu \Vert_{2m}\leq C(m^{1/2}\Vert \nabla\uu\Vert_{2}+1),$ for any $m\geq 1$;
\item[(3)] $\Vert \nabla\xi\Vert_{2m}\leq C(m^{3/2}k^{1/2}\Vert \rho\Vert_{\frac{2mk}{k-1}}\phi(t)^{1/2}+m\Vert \rho\Vert_{\frac{2mk}{k-1}})$ for any $m\geq 1$ and $k>1$;
\item[(4)] $\Vert \eta\Vert_{2m}\leq C(m^2k\Vert \rho\Vert_{\frac{2mk}{k-1}}+m\Vert\rho\Vert_{\frac{2mk}{k-1}}+m^2\phi(t)+m),$ for any $m\geq 1$ and $k>1$,
\end{enumerate}
\vskip 5pt
where $C$ is a positive constant independent of $m$ and $k$, and 
\[
\phi(t):=\int_\Omega (\mu\omega^2 + (2\mu+\lambda(\rho))(\Div\uu)^2 + \nu|\nabla\HH|^2)d\x,\qquad t\in[0,T].
\]
\end{lemma}

Using the transport equation \eqref{transportrho} and the elliptic estimates above, as well as Lemma~\ref{lemmaHLp}, we can derive the following $L^p$ estimate for the density, which corresponds to the analogue estimate in \cite[lemma 3.5]{Mei}.
\begin{lemma}\label{rhop}
Assume $\beta>1$. Then, for any $p\geq 1$,
\begin{equation}
\sup_{0\leq t \leq T} \Vert \rho(t) \Vert_{L^p(\Omega)}\leq Cp^{\frac{2}{\beta-1}},
\end{equation}
where $C$ is a positive constant independent of $p$.
\end{lemma}

\begin{proof}
Fix some $m\in\mathbb{N}$ and multiply equation \eqref{transportrho} by $\rho[(\Lambda(\rho)-\xi)_+]^{2m-1}$, where $(\cdot)_+$ denotes the positive part. Then, integrating the resulting equation over $\Omega$ we have
\begin{flalign}
&\frac{1}{2m}\frac{d}{dt}\int \rho[(\Lambda(\rho)-\xi)_+]^{2m}d\x\label{ineqrhop}&\\
&\qquad+\int \rho P [(\Lambda(\rho)-\xi)_+]^{2m-1}d\x + \frac{1}{2}\int \rho|\HH|^2[(\Lambda(\rho)-\xi)_+]^{2m-1}d\x&\nonumber\\
&=\alpha\int \rho g'(1/ \rho)h(|\psi\circ \Y|^2)\frac{\J_\y}{\rho}[(\Lambda(\rho)-\xi)_+]^{2m-1}d\x &\nonumber\\
&\qquad + \int \rho\eta[(\Lambda(\rho)-\xi)_+]^{2m-1}d\x-\int \rho\uu\cdot\nabla\xi[(\Lambda(\rho)-\xi)_+]^{2m-1}d\x& \nonumber\\
&\qquad\qquad + \left(\int \F(\x,t)d\x\right)\left(\int \rho[(\Lambda(\rho)-\xi)_+]^{2m-1}d\x \right):=\sum_{i=0}^3 K_i.& \nonumber
\end{flalign}

Let us denote 
\[
f(t):=\left( \int \rho[(\Lambda(\rho)-\xi)_+]^{2m}d\x\right)^{1/2m},
\]
so that
\begin{equation}
\Vert \rho(t) \Vert_{2m\beta + 1}^\beta\leq C\left[1+f(t)+\left(\int \rho|\xi|^{2m}d\x \right)^{\frac{\beta}{2m\beta+1}}  \right].\label{rholeqf}
\end{equation}

Note that, since $g'$ and $h$ are bounded functions, using \eqref{J/rho} the term $K_0$ can be easily bounded as
\begin{align}
|K_0| &\leq C\int \rho^{1/2m}[\rho(\lambda(\rho)-\xi)_+^{2m}]^{\frac{2m-1}{2m}}d\x\label{K_0}\\
& \leq C\Vert \rho \Vert_{1}^{1/2m}\Vert \rho(\lambda(\rho)-\xi)_+^{2m}\Vert_{1}^{\frac{2m-1}{2m}}\nonumber\\
& \leq C f^{2m-1}(t).\nonumber
\end{align}

Next, just as in the proof of \cite[lemma 3.5]{Mei}, using lemma \ref{elliptic2}, we have that
\begin{equation}
|K_1|\leq C\Vert \rho \Vert_{2m\beta+1}^{1+\frac{1}{2m}}f^{2m-1}(t)[m^2\phi(t)+m],\label{K_1}
\end{equation}
and
\begin{equation}
|K_2|\leq C\Vert \rho \Vert_{2m\beta+1}^{\frac{1}{2m}}f^{2m-1}(t)[m^2\phi(t) + m].\label{K_2}
\end{equation}

Regarding $K_3$ we note that, similarly to \eqref{K_0} we have that
\[
\int \rho[(\Lambda(\rho)-\xi)_+]^{2m-1}d\x \leq C f^{2m-1}(t),
\]
while 
\begin{align*}
\int |\F |d\x &\leq \left(\int(2\mu + \lambda(\rho))(\Div\uu)^2d\x\right)^{1/2}\left(\int(2\mu+\lambda(\rho))d\x\right)^{1/2} \\
&\qquad+ \int Pd\x + \frac{1}{2}\int|\HH|^2d\x + \alpha\int\left|g'(1/\rho)h(|\psi\circ\Y|^2)\frac{\J_\y}{\rho}\right|d\x\\
 &\leq C\left[\phi^{1/2}(t) + \phi^{1/2}(t)\left(\int \rho^\beta d\x\right)^{1/2} +1 \right]\\
 &\leq C [\phi^{1/2}(t) + \phi^{1/2}(t)\Vert \rho \Vert_{2m\beta+1}^{\frac{\beta}{2}}+1].
\end{align*}
Therefore,
\begin{equation}
|K_3|\leq Cf^{2m-1}(t)[\phi^{1/2}(t) + \phi^{1/2}(t)\Vert \rho \Vert_{2m\beta+1}^{\frac{\beta}{2}}+1].\label{K_3}
\end{equation}

Substituting  \eqref{K_0}-\eqref{K_3} into \eqref{ineqrhop} we have
\begin{equation}
\frac{d}{dt}f(t)\leq C[1+\phi^{1/2}(t) + \phi^{1/2}(t)\Vert \rho \Vert_{2m\beta+1}^{\frac{\beta}{2}}+(m^2\phi(t) + m)\Vert \rho \Vert_{2m\beta+1}^{1+\frac{1}{2m}}].\label{ddtf}
\end{equation} 

Once we have inequalities \eqref{rholeqf} and \eqref{ddtf} the remaining of the proof follows from lemma \ref{elliptic2} and Gronwall's inequality, wherein the details follow line by line the proof of \cite[lemma 3.5]{Mei}.
\end{proof}

\begin{lemma}\label{Zvarphi}
For any $\varepsilon>0$ there exists a positive constant $C(\varepsilon)$ such that
\begin{equation}
\sup_{0\leq t\leq T} \log (e+Z^2(t))+\int_0^T\frac{\varphi^2(t)}{e+Z^2(t)}dt\leq C\Phi_T^{1+\beta\varepsilon},
\end{equation}
where $Z^2(t):=\int (\mu\omega^2+\frac{\F^2}{2\mu+\lambda(\rho)}+|\nabla \HH|^2)d\x,$ $\varphi^2(t):=\int(\rho|\dot{\uu}|^2+|\nabla^2 \HH|^2)d\x$ and $\Phi_T:=||\rho||_\infty+1$.

In particular,
\begin{equation}
\sup_{0\leq t\leq T}\log(e+\phi(t))\leq C\Phi_T^{1+\beta\varepsilon},\label{suplogphi}
\end{equation}
where, $\phi$ is as defined in Lemma~\ref{elliptic2}.
\end{lemma}

\begin{proof}
The proof of this lemma consists of bounding appropriately all the terms in the right hand side of the following identity
\begin{align}
&\frac{d}{dt}Z^2 + 2\varphi^2 = -2\frac{d}{dt}\int_\Omega \HH\otimes\HH :\nabla\uu d\x - \mu\int_\Omega \omega^2\Div \uu d\x \label{idZ}\\
&\quad + 4\int_\Omega \F \nabla u_1\cdot\nabla^\perp u_2 d\x + \int_\Omega \F^2\Div \uu \left[ \rho\left( \frac{1}{2\mu+\lambda(\rho)}\right)' - \frac{1}{2\mu+\lambda(\rho)} \right] d\x \nonumber\\
&\quad +2\int_\Omega \F\Div \uu \left[ \rho\left( \frac{P}{2\mu+\lambda(\rho)} \right)' - \frac{P}{2\mu+\lambda(\rho)} \right] d\x \nonumber\\
&\quad +\int_\Omega \F |\HH|^2\Div\uu \left[ \rho\left( \frac{1}{2\mu+\lambda(\rho)}\right)' - \frac{1}{2\mu+\lambda(\rho)} \right] d\x \nonumber\\
&\quad -2\int_\Omega\frac{\F}{2\mu+\lambda(\rho)}(\HH\cdot\nabla\uu\cdot\HH + \nu\HH\cdot\Delta\HH)d\x &\nonumber\\
&\quad+2\int_\Omega(\nu\Delta\HH - \HH\Div\uu)\cdot\nabla\uu\cdot\HH d\x + 2\int_\Omega\HH\cdot\nabla\uu(\nu\Delta\HH+\HH\cdot\nabla\uu) \nonumber\\
&\quad +2\int_\Omega (\uu\cdot\nabla\HH - \HH\cdot\nabla\uu + \HH\Div\uu)\cdot\Delta\HH d\x \nonumber\\
&\quad-2\alpha\int_\Omega \F g'(1/\rho)h(|\psi\circ\Y|^2)\frac{\J_\y}{\rho}\Div\uu \left[ \rho\left( \frac{1}{2\mu+\lambda(\rho)}\right)' - \frac{1}{2\mu+\lambda(\rho)} \right]d\x \nonumber\\
&\quad +2\alpha \int_\Omega \frac{\F}{2\mu+\lambda(\rho)}\Div\uu\frac{1}{\rho}g''(1\rho)h(|\psi\circ\Y|^2)\frac{\J_\y}{\rho}d\x \nonumber\\
&\quad +2\alpha\int_\Omega \frac{\F}{2\mu+\lambda(\rho)}D_t(|\psi\circ\Y|^2)g'(1/\rho)h'(|\psi\circ\Y|^2)\frac{\J_\y}{\rho}  \nonumber\\
&\quad:=-\frac{d}{dt}J_0 + \sum_{i=1}^{12}J_i.\nonumber
\end{align}
Here $D_t=\frac{d}{dt}+\uu\cdot\nabla$ stands for the material derivative and $(\cdot)'$ stands for differentiation with respect to $\rho$.

This identity follows after some lengthy and tedious manipulation by rewriting the momentum equation as
\begin{equation}
\rho\dot{\uu} = \nabla\F +\mu\nabla^\perp\omega + \HH\cdot\nabla\HH,\label{E2dotu}
\end{equation}
where $\dot{\uu}=D_t\uu$, multiplying by $\dot{\uu}$, integrating by parts and using the continuity equation \eqref{E2rho} in order to deal with the material derivative of some of the terms. We write $\Div\uu$ as 
\begin{equation}
\Div\uu=\frac{\F + P(\rho)+\tfrac{1}{2}|\HH|^2-\alpha g'(1/\rho)h(|\psi\circ\Y|^2)\tfrac{\J_\y}{\rho}}{2\mu+\lambda(\rho)},\label{uitoF}
\end{equation}
in order to complete the term $\frac{\F^2}{2\mu+\lambda(\rho)}$ on the left hand side and then we use the magnetic field equation \eqref{E2H} to transform some troublesome terms into ones that are easier to handle. In particular, we use the following key observation
\begin{flalign*}
&\int_\Omega \HH\cdot\nabla\HH\cdot\dot{\uu}d\x=-\frac{d}{dt}\int_\Omega \HH\otimes\HH:\nabla\uu d\x + \int_\Omega(\nu\Delta\HH - \HH\Div\uu)\cdot\nabla\uu\cdot\HH d\x&\\
&\qquad\qquad\qquad\qquad\qquad\qquad+\int_\Omega\HH\cdot\nabla\uu\cdot(\nu\nabla\HH + \HH\cdot\nabla\uu)d\x.
\end{flalign*}
Another important observation in the deduction of identity \eqref{idZ} is that by virtue of \eqref{J/rho=const} we have that $D_t\left( \frac{\J_\y}{\rho}\right)=0$.

This identity has an analogue in \cite{Mei} where the terms $J_0,...,J_9$ are identical and, once we have lemma \ref{rhop}, they may be estimated as
\[
\sum_{i=1}^9|J_i|\leq C\Phi_T^{1+\beta\varepsilon}(1+\|\nabla\uu\|_2^2+\|\nabla\HH \|_2^2)(1+Z^2(t)),
\]
for any $\varepsilon>0$, where $C=C(\varepsilon)$ is a positive constant.

Moreover, noting that $\|\nabla \uu(t)\|_2^2\leq C(1+Z^2(t))$ (see, for instance, \cite[lemma 1]{P}) and using Lemma \ref{lemmaHLp} along with Young's inequality with $\varepsilon$, we have
\begin{equation}
|J_0|\leq 2\|\HH \|_4^2\| \nabla\uu \|_2\leq C_0 + \frac{1}{2}Z^2\label{J_0}
\end{equation}

In our case, the terms $J_{10}$, $J_{11}$ and $J_{12}$ appear due to the short wave-long wave interactions. From the fact that $g'$ has compact support in $(0,\infty)$ and using \eqref{J/rho} as well as Lemma~\ref{aprioripsi} we have
\begin{align*}
\sum_{i=10}^{12}|J_i|&\leq C\left\| \frac{\F}{(2\mu+\lambda(\rho))^{1/2}}\right\|_2 \left[ \|\nabla\uu\|_2+\int_\Omega |D_t(\psi\circ\Y)|^2 \J_\y d\x \right]\\
  &= C\left\| \frac{\F}{(2\mu+\lambda(\rho))^{1/2}}\right\|_2 \left[ \|\nabla\uu\|_2+\int_{\Omega_\y} |\psi_t|^2 d\y \right]\\
  &\leq C Z(t)[1+\|\nabla\uu\|_2]\\
  &\leq C(1+Z^2(t))(1+\|\nabla\uu\|_2^2).
\end{align*}

Thus, we conclude that 
\begin{equation}
\frac{d}{dt}(1+Z^2(t)-J_0)+\varphi^2(t)\leq C\Phi_T^{1+\beta\varepsilon}(1+\|\nabla\uu\|_2^2 + \|\nabla\HH\|_2^2)(1+Z^2(t))\label{Zvarphidif}
\end{equation}
and, by virtue of \eqref{J_0}, Lemma~\ref{apriorienergy} and Gronwall's inequality yield the result.
\end{proof}

Before we prove the key uniform estimates for the density we need the following lemma whose proof is identical to \cite[lemma 0.1]{Mei'}.

\begin{lemma}\label{rhoup}
For any $p>4$, there exists a positive constant $C(p)$ such that
\[
\|\rho\uu\|_p\leq C(p)\Phi_T^{1+\frac{\beta}{4}}(\|\nabla\uu\|_2+1)^{1-\frac{2}{p}}.
\]
\end{lemma}

We can now prove that neither vacuum, nor concentration develop in finite time, which guarantees that the Lagrangian transformation is well defined at all times.
\begin{lemma}\label{aprioriunifrho}
Let $\beta> \frac{4}{3}$. Then, there exists a constant $C>0$ such that 
\begin{equation}
C^{-1}\leq\rho(x,t)\leq C,\qquad\text{for all }(x,t)\in\Omega\times[0,T].
\end{equation}

In particular,
\begin{equation}
\sup_{0\leq t\leq T}\int_\Omega (|\nabla \uu|^2 + |\nabla\HH|^2)d\x +\int_0^T \varphi^2(t)dt \leq C.
\end{equation}
\end{lemma}

\begin{proof}
As outlined in the beginning of this section, the proof consists in analysing equation \eqref{transportrho}. First, similarly as in \cite{P,Mei} (cf. \cite{H,Li}), from \eqref{defxi} and \eqref{defeta} we note that we can write
\[
u\cdot \nabla\xi-\eta = [u_i,R_iR_j](\rho u_j)-[H_i,R_iR_j](H_j),
\]
where $R_i$ are the Riesz transform and $[b,R_i R_j](f)=bR_iR_j(f) - R_iR_j(bf)$ is the Lie bracket.

Then, we can rewrite equation \eqref{transportrho} in the Lagrangian coordinates as
\begin{equation}
\begin{aligned}
&\Big(\Lambda(\rho(\y,\tau))-\xi(\y,\tau)\Big)_t + P(\rho(\y,\tau))+\frac{|\HH|^2}{2}(\y,\tau) - \alpha g'(v)h(|\psi|^2)v\J_y(\y,\tau) \\
&\qquad +[u_i,R_iR_j](\rho u_j(\y,\tau))-[H_i,R_iR_j](H_j)(\y,\tau) + \int_\Omega \F(\tilde{\x},\tau)d\tilde{\x} = 0.
\end{aligned}\label{transporrhodif}
\end{equation}
Thus, upon integrating over $[0,t]$, we conclude that, for any $(\x,t)\in\Omega\times [0,T]$ we have
\begin{equation}
\begin{aligned}
&2\mu \log\frac{\rho(\x,t)}{\rho_0(\x_0)}+\frac{1}{\beta}(\rho^\beta(\x,t)-\rho_0^\beta(\x_0)) - \xi(\x,t)+\xi_0(\x_0)\\
& \leq C+ \int_0^t [u_i,R_iR_j](\rho u_j)(\tau) d\tau + \int_0^t [H_i,R_iR_j](H_j)(\tau)d\tau -\int_0^t\int_\Omega \F(\tilde{\x},\tau)d\tilde{\x} d\tau,\\
\end{aligned}\label{uppertransportrho}
\end{equation}
where $\x_0 = \y(\x,t)$. Note that we used \eqref{J/rho} as well as the uniform bounds on $g'$ and $h$.

Let us recall the result by Coiffman et al. \cite{CRW,CM,CLiM} stating that, for any $b,f\in C^\infty(\Omega)$, then for $p\in(1,\infty)$ there exists a constant $C(p)>0$ such that
\begin{equation}\label{coifetal1}
\| [b,R_i R_j](f)\|_p\leq C \| b\|_{BMO}\|f\|_p.
\end{equation}
Moreover, if $q_k\in (1,\infty)$, $k=1,2,3$ satisfy $\frac{1}{q_1}=\frac{1}{q_2}+\frac{1}{q_3}$, then there is a constant $C$ depending on $q_k$ such that
\begin{equation}\label{coifetal2}
\| \nabla [b,R_i R_j](f)\|_{q_1}\leq \|\nabla b\|_{q_2}\|f\|_{q_2}.
\end{equation}

With this result at hand, we can bound the first two integrals on the right hand side of \eqref{uppertransportrho}. Indeed, for $p>4$,  the following Gagliardo-Nirenberg type inequality follows from a well known inequality (see, e.g., \cite{A} theorem~5.9, p.140) 
\begin{align}
\| [u_i,R_iR_j](\rho u_j) \|_\infty &\le C\| [u_i,R_iR_j](\rho u_j) \|_p^{1-\frac{4}{p}}\|\nabla [u_i,R_iR_j](\rho u_j)\|_{\frac{4p}{p+4}}^{\frac{4}{p}} \label{adams}\\
\intertext{and, using \eqref{coifetal1}, \eqref{coifetal2}, we have that the right hand-side can be estimated by}
  &\leq C\left(\|\uu\|_{BMO}\|\rho\uu\|_p \right)^{1-\frac{4}{p}}\left( \|\nabla\uu\|_4 \|\rho\uu\|_p \right)^{\frac{4}{p}}\nonumber\\
  &\leq C \Phi_T^{1+\frac{\beta}{4}}\left( \|\nabla\uu\|_2 + 1 \right)^{2-\frac{6}{p}}\|\nabla\uu\|_4^{\frac{4}{p}},\label{commut1}
\end{align}
where the last inequality follows from lemma \ref{rhoup}.

On the other hand, similarly as in \cite{Mei}, using the following elliptic estimate
\begin{equation}
\|\nabla\uu\|_4\leq C(\|\Div\uu\|_4+\|\omega\|_4),\label{nablau<=divucurlu}
\end{equation}
and writing $\Div\uu$ as in \eqref{uitoF}, after some manipulation, for any $\varepsilon>0$ we have that
\begin{equation}
\begin{aligned}
\|\nabla \uu\|_4&\leq C\Phi_T^{\frac{1+\beta\varepsilon+2\beta}{4}}(e+\|\nabla\uu\|_2+\|\nabla\HH\|_2)\left( \frac{\varphi^2}{e+Z^2} \right)^{\frac{1}{4}}\\
   &\qquad +C\Phi_T^{\frac{2\beta\varepsilon+5\beta}{8}}(e+\|\nabla\uu\|_2+\|\nabla\HH\|_2)^{\frac{5}{4}}\left( \frac{\varphi^2}{e+Z^2} \right)^{\frac{1}{8}} +C,
\end{aligned}\label{commut2}
\end{equation}
where $Z$ and $\varphi$ are those defined in lemma \ref{Zvarphi}. Here we have used the fact that
\[
Z^2\leq C\Phi_T^\beta \left( \|\nabla\uu\|_2^2 + \|\nabla\HH\|_2^2 \right)
\] 
and that for any $\varepsilon>0$
\begin{align*}
\left\| \frac{\F}{2\mu+\lambda(\rho)} \right\|_4^2&\leq C\left\| \frac{\F}{\sqrt{2\mu+\lambda(\rho)}} \right\|_2^{1-\varepsilon}\|\F \|_{\frac{2(1+\varepsilon)}{\varepsilon}}^{1+\varepsilon}\leq CZ^{1-\varepsilon}\|\F \|_2^\varepsilon \|\nabla F\|_2\\
  &\leq CZ \Phi_T^{\frac{\beta\varepsilon}{2}}\left( \Phi_T^{\frac{1}{2}}\varphi + \|\nabla\HH \|_2\varphi^{\frac{1}{2}} \right),
\end{align*}
where, the last inequality follows by a standard elliptic estimate on the identity $\Delta F=\Div(\rho\dot{\uu}-\HH\cdot\nabla\HH)$, which holds by virtue of \eqref{E2dotu}.

Thus, putting \eqref{commut1}  and \eqref{commut2} together and choosing $p>4$ large enough we have that
\begin{equation}
\int_0^T \| [u_i,R_iR_j](\rho u_j) \|_\infty dt\leq C\Phi^{1+\frac{\beta}{4}+\beta\varepsilon}.
\end{equation}

Similarly, using Lemma~\ref{lemmaHLp}, for $p>4$, again using the analogue of \eqref{adams} and \eqref{coifetal1}, \eqref{coifetal2},   we have that
\begin{align*}
\| [H_i,R_iR_j](H_j) \|_\infty & \leq C \left( \| \HH\|_{BMO}\| \HH\|_p \right)^{1-\frac{4}{p}}\left( \|\nabla\HH\|_4\|\HH\|_p \right)^{\frac{4}{p}}\\
   &\leq C\|\nabla\HH\|_2^{1-\frac{4}{p}}\| \nabla\HH\|_4^{\frac{4}{p}}\leq C\|\nabla\HH\|_2^{1-\frac{2}{p}}\| \nabla^2\HH\|_2^{\frac{2}{p}}\\
   &\leq C\Phi_T^{\frac{\beta}{2}}(\| \nabla\uu \|_2+\|\nabla\HH\|_2)\left( \frac{\varphi^2(t)}{e+Z^2(t)} \right)^{\frac{1}{p}}.
\end{align*}

Thus, by virtue of Lemma~\ref{Zvarphi} we obtain that
\begin{equation}
\int_0^T\| [H_i,R_iR_j](H_j) \|_\infty dt\leq C\Phi_T^{1+\beta\varepsilon}.
\end{equation}

Now, we note that from Lemmas~\ref{apriorienergy} and \ref{rhop} we have that
\[
\int_0^T \F(\tilde{\x},t)dx dt \leq C.
\]
Moreover, for a suitably large but fixed $m>1$, Lemma~\ref{elliptic1} implies that
\[
\| \xi \|_{2m} \leq Cm^{1/2}\|\rho \|_m^{1/2}\leq C,\qquad \|\nabla\xi\|_{2}\leq \|\rho\uu \|_2 \leq C\Phi_T^{\frac{1}{2}}.
\]
Also, Lemma~\ref{elliptic2} implies that
\[
\| \nabla\xi\|_{2m}\leq C(1+\phi(t)^{\frac{1}{2}}).
\]

Now, we recall the inequality by Br\'{e}zis and Wainger (see \cite{BrWa,En}) stating that for the $2$-torus $\Omega$ (and also for $\mathbb{R}^2$),  for any $q>2$, there is a positive constant $C$ such that
\[
\|w\|_{L^\infty(\Omega)}\leq C\| \nabla w\|_{L^2(\Omega)}\sqrt{\log(1+\|\nabla w \|_{L^q(\Omega)}}+C\| w\|_{L^2(\Omega)} + C,
\]
for any $w\in W^{1,q}(\Omega)$.

Hence, in light of Lemma~\ref{Zvarphi}, we have
\begin{align*}
\| \xi\|_\infty&\leq C(\|\xi\|_{2m}+\|\nabla\xi \|_2)\log^{\frac{1}{2}}(e+\| \xi\|_{W^{1,2m}})\leq \Phi_T^{\frac{1}{2}}\log^{\frac{1}{2}}(e+\| \xi\|_{W^{1,2m}})\\
  &\leq C\Phi_T^{\frac{1}{2}} \log^{1/2}(e+\phi(t))\leq C\Phi_T^{1+\frac{\beta\varepsilon}{2}}.
\end{align*}

Gathering all this information in \eqref{uppertransportrho} we obtain that,
\[
\Phi_T^\beta\leq C\Phi_T^{1+\frac{\beta}{4}+\beta\varepsilon},
\]
so that, if $\beta>\frac{4}{3}$, we can choose $\varepsilon>0$ small enough and conclude that
\[
\sup_{(\x,t)\in\Omega\times[0,T]}\rho(\x,t) \leq C.
\]
In particular, we can apply Gronwall's inequality in \eqref{Zvarphidif} to conclude that
\[
\int_0^T \varphi^2(t) dt\leq C,
\] 
which, implies that
\[
\int_0^T \| \HH \|_\infty^2 dt\leq C.
\]
Taking this information back to \eqref{transporrhodif} we can also derive a positive lower bound for the density by a similar reasoning. That is,
\begin{equation}
\inf_{(\x,t)\in\Omega\times[0,T]}\rho(\x,t) \geq C^{-1}.
\end{equation}
\end{proof}

We can summarize the a priori estimates obtained so far in the following lemma.
\begin{lemma}\label{aprioriloworder}
There is a constant $C>0$ depending only on $T$, $\| \uu_0\|_{H^1(\Omega)}$, $\| \HH_0\|_{H^1(\Omega)}$, $\|\psi_0\|_{H^2(\Omega_\y)}$ and the uniform bounds from above and below on $\rho_0$ such that
\begin{equation}
C^{-1}\leq \rho (\x,t)\leq C,\qquad (\x,t)\in \Omega\times [0,T],
\end{equation}
\begin{equation}
\sup_{t\in[0,T]}(\| \uu(t)\|_{H^1(\Omega)} + \| \HH(t)\|_{H^1(\Omega)}) + \int_0^T (\|\sqrt{\rho}\dot{\uu}\|_{L^2(\Omega)}^2 + \|\nabla^2\HH\|_{L^2(\Omega)}^2)ds\leq C,
\end{equation}
\begin{equation}
\sup_{t\in[0,T]}(\| \psi_t(t) \|_{L^2(\Omega_\y)}+\|\psi(t)\|_{H^2(\Omega_\y)})\leq C.
\end{equation}
\end{lemma}

\subsection{$H^2$ estimates}

Now, we turn our attention to the a priori estimates in $H^2$ for the density, velocity and magnetic field.
\begin{lemma}\label{nablaudotut}
There is a constant $C>0$, depending on $\| \uu_0\|_{H^2(\Omega)}$, $\| \HH_0\|_{H^2(\Omega)}$, $\|\psi_0\|_{H^2(\Omega_\y)}$ and the uniform bounds from above and below on $\rho_0$, such that
\begin{equation}
\sup_{0\leq t\leq T} (\|\sqrt{\rho}\dot{\uu}\|_{L^2(\Omega)}^2+\|\HH_t\|_{L^2(\Omega)}^2)+\int_0^T(\|\nabla\dot{\uu}\|_{L^2(\Omega)}^2+\|\nabla\HH_t\|_{L^2(\Omega)^2}^2dt)\leq C. \label{rhodotuHt}
\end{equation}
\end{lemma}

\begin{proof}
Let us denote $f=\alpha g'(1/\rho)h(|\psi\circ\Y|^2)\frac{J_\y}{\rho}$. Applying the operator $\dot{\uu}^j[\partial_t + \Div(\uu\cdot)]$ to the $j^{th}$ component of the momentum equation equation $\eqref{E2u}$, summing with respect to $j$ and integrating we obtain
\begin{align}
\frac{1}{2}\frac{d}{dt}\int_\Omega \rho |\dot{\uu}|^2d\x &= \int_\Omega \dot{\uu}^j[\partial_{x_j}f_t + \Div(\uu \partial_{x_j}f)]d\x - \int_\Omega \dot{\uu}^j[\partial_{x_j}p_t + \Div(\uu \partial_{x_j}p)]d\x\nonumber\\
  &+\mu \int_\Omega \dot{\uu}^j[\partial_t \Delta\uu^j + \Div(\uu \Delta\uu^j)]d\x \nonumber \\
  &+ \int_\Omega \dot{\uu}^j[\partial_{x_jt}((\mu+\lambda)\Div\uu) + \Div(\uu \partial_{x_j}((\mu+\lambda)\Div\uu))]d\x\nonumber\\
  &+\int_\Omega \dot{\uu}^j[\partial_t ((\HH\cdot\nabla)\HH^j) + \Div(\uu (\HH\cdot\nabla)\HH^j)]d\x\nonumber\\
  &+\frac{1}{2}\int_\Omega \dot{\uu}^j[\partial_{x_jt}(|\HH|^2) + \Div(\uu \partial_{x_j}(|\HH|^2))]d\x:=\sum_{i=0}^5 N_i.
\end{align}

Again, there is analogue identity in \cite{Mei}, wherein the terms $N_1,\cdots,N_5$ are identical and, through integration by parts and some manipulation using the continuity equation, are estimated as
\begin{align*}
&\sum_{i=1}^5N_i \leq -\frac{\mu}{2}( \|\nabla \dot{\uu}\|_{L^2(\Omega)}^2 + \| D_t\Div\uu\|_{L^2(\Omega)}^2) + \delta \|\nabla\HH_t\|_{L^2(\Omega)}^2\\
&\qquad\qquad\qquad + C_\delta (1+\|\HH_t\|_{L^2(\Omega)}^2+  \|\sqrt{\rho}\dot{\uu}\|_{L^2(\Omega)}^2 +\|\nabla\uu\|_{L^4(\Omega)}^4),
\end{align*}
for arbitrary $\delta>0$ (to be chosen) and some suitably large $C_\delta$.

Regarding the term $N_0$, due to the short wave-long wave interactions, integrating by parts, we see that
\begin{align*}
N_0 = \int_\Omega\Div\dot{\uu} (f_t+\Div(f\uu))-2\int_\Omega f\partial_{x_j}\uu^k \partial_{x_k}\dot{\uu}^j.
\end{align*}
Moreover, using the continuity equation, we have that
\begin{align*}
f_t+\Div(f\uu)&=D_t f + f\Div\uu\\
  &=-\alpha \frac{g''(1/\rho)}{\rho}h(|\psi\circ\Y|^2)\frac{J_\y}{\rho} + \alpha g'(1/\rho)D_t[h(|\psi\circ\Y|^2)] \frac{J_y}{\rho}\\
  &\qquad\qquad + f\Div\uu,
\end{align*}
and, noting that $\int_\Omega |D_t(\psi\circ\Y)|^2J_\y d\x = \int_{\Omega_\y}|\psi_t|^2 d\y\leq C$, using Young's inequality, we obtain that
\[
N_0\leq \frac{\mu}{8}\|\nabla\dot{\uu}\|_{L^2(\Omega)}^2 + C.
\]

On the other hand, applying the operator $\partial_t$ to the magnetic field equation \eqref{E2H}, multiplying the resulting equation by $\HH_t$ and integrating we obtain
\begin{align*}
&\frac{1}{2}\frac{d}{dt}\int_\Omega |\HH_t|^2d\x + \nu\int_\Omega |\nabla\HH_t|^2d\x\\
&=-\int_\Omega \uu_t\cdot\nabla\HH\cdot\HH_td\x + \int_\Omega \HH_t\cdot\nabla\uu\cdot\HH_td\x + \int_\Omega \HH\cdot\nabla\uu_t\cdot\HH_t d\x\\
&\qquad -\int_\Omega |\HH_t|^2\Div\uu d\x - \int_\Omega\HH\cdot\Div\uu_t\HH_t d\x\\
&= \int_\Omega (\HH\cdot\nabla\dot{\uu}-\dot{\uu}\nabla\HH-\HH\Div\dot{\uu})\cdot\HH_td\x + \int_\Omega (H^i\partial_iH_t^j-H^i\partial_jH_t^i)(\uu\cdot\nabla u^j)d\x\\
&\qquad +\int_\Omega \left(\HH_t\cdot\nabla\uu - \frac{1}{2}\HH_t\Div\uu-\uu\cdot\nabla\HH_t \right)\HH_td\x\\
&\leq \delta \|\nabla\dot{\uu}\|_{L^2(\Omega)}^2 + \frac{\nu}{2}\|\nabla\HH_t\|_{L^2(\Omega)}^2 \\
&\qquad+C_\delta(1+\|\HH_t\|_{L^2(\Omega)}^2+\|\sqrt{\rho}\dot{\uu}\|_{L^2(\Omega)}^2+\|\nabla\uu\|_{L^4(\Omega)}^4).
\end{align*}

Gathering all this information and choosing $\delta>0$ small enough we arrive at
\begin{align*}
&\frac{1}{2}\frac{d}{dt}(\|\sqrt\rho \dot{\uu}\|_{L^2(\Omega)}^2+\|\HH_t\|_{L^2(\Omega)}^2)  + \|\nabla\dot{\uu}\|_{L^2(\Omega)}^2+\|\nabla\HH_t\|_{L^2(\Omega)}^2\\
&\qquad\qquad\qquad\qquad\qquad\leq C(1+\|\sqrt\rho \dot{\uu}\|_{L^2(\Omega)}^2+\|\HH_t\|_{L^2(\Omega)}^2+\|\nabla\uu\|_{L^4(\Omega)}^4).
\end{align*}
Finally, using \eqref{nablau<=divucurlu} and standard elliptic estimates on the identities 
\begin{equation}
\Delta F=\Div(\rho\dot{\uu}-\HH\cdot\nabla\HH),\hspace{5mm} \text{ and }\hspace{5mm}\mu\Delta\omega=\nabla^\perp\cdot(\rho\dot{\uu}-\HH\cdot\nabla\HH),\label{Fomegaelliptic}
\end{equation} 
which hold by virtue of \eqref{E2dotu} we see that
\begin{align*}
\|\nabla\uu\|_{L^4(\Omega)}^4&\leq C(1+\|\F\|_{L^4(\Omega)}^4+\|\omega\|_{L^4(\Omega)}^4)\\
  &\leq C(1+\|\F\|_{L^2(\Omega)}^2\|\nabla \F\|_{L^2(\Omega)}^2+\|\omega\|_{L^2(\Omega)}^2\|\nabla\omega\|_{L^2(\Omega)}^2)\\
  &\leq C(1+\|\sqrt\rho \dot{\uu}\|_{L^2(\Omega)}^2+\|\nabla^2\HH\|_{L^2(\Omega)}^2).
\end{align*}

Thus, using Gronwall's inequality we obtain \eqref{rhodotuHt}.
\end{proof}

\begin{lemma}\label{lemmanablarho}
Let $2\leq p < \infty$ and let $\E(t,\x)=\nabla_\x \y(t,\x)$ be the deformation gradient. Then,
\begin{equation}
\sup_{0\leq t\leq T}(\|\nabla\rho\|_{L^p(\Omega)} + \|\E\|_{L^{2p}(\Omega)})+\int_0^T \|\nabla\uu\|_{L^\infty}^2dt \leq C.
\end{equation}
\end{lemma}

\begin{proof}
First, we note that from the estimates from Lemmas~\ref{aprioriloworder} and \ref{nablaudotut} we have
\[
\|\rho\dot{\uu}\|_{L^p(\Omega)}+\||\HH||\nabla\HH|\|_{L^p(\Omega)}\leq C(1+\|\nabla\dot{\uu}\|_{L^2(\Omega)}+\|\nabla^2\HH\|_{L^2(\Omega)}).
\] 

As a consequence, from the standard elliptic estimates on identities \eqref{Fomegaelliptic}, we have
\begin{align*}
\|\Div\uu\|_{L^\infty(\Omega)}+\|\omega\|_{L^\infty(\Omega)}&\leq C(\|\F\|_{L^\infty(\Omega)}+\|\omega\|_{L^\infty(\Omega)}+\|\HH\|_{L^\infty(\Omega)}^2+1)\\
 &\leq C(\|\nabla\F\|_{L^4(\Omega)}^{\frac{3}{2}}+\|\nabla\omega\|_{L^4(\Omega)}^{\frac{3}{2}}+\|\nabla\HH\|_{L^4(\Omega)}^{\frac{3}{2}}+1)\\
 &\leq C(\|\nabla\dot{\uu}\|_{L^2(\Omega)}^{\frac{3}{2}}+\|\nabla^2\HH\|_{L^2(\Omega)}^{\frac{3}{2}}+1)
\end{align*}

Next, applying the operator $\nabla$ to the continuity equation \eqref{E2rho}, multiplying the resulting equation by $p|\nabla\rho|^{p-2}\nabla\rho$ and integrating we have
\begin{align*}
\frac{d}{dt}\int_\Omega|\nabla\rho|^pd\x&=-(p-1)\int_\Omega |\nabla\rho|^p\Div\uu d\x - p\int_\Omega |\nabla\rho|^{p-2}\nabla\rho\cdot\nabla\uu\cdot\nabla\rho d\x\\
&\qquad -p\int_\Omega \rho|\nabla\rho|^{p-2}\nabla\rho\cdot\nabla\Div\uu d\x,
\end{align*}
so that
\begin{equation}
\frac{d}{dt}\|\nabla\rho\|_{L^p(\Omega)} \leq C(\|\nabla\uu\|_{L^\infty(\Omega)}\|\nabla\rho\|_{L^p(\Omega)}+\|\nabla^2\uu\|_{L^p(\Omega)}).\label{nablarhopprelim}
\end{equation}

Denoting $f=\alpha g'(1/\rho)h(|\psi\circ\Y|^2)\frac{J_\y}{\rho}$ as in the proof of Lemma~\ref{nablaudotut} we have
\begin{align}
\|\nabla^2\uu\|_{L^p(\Omega)}&\leq C(\|\Div\uu\|_{L^p(\Omega)}+\|\nabla\omega\|_{L^p(\Omega)})\nonumber\\
 &\leq C(\|\nabla\F\|_{L^p(\Omega)}+\|\nabla\omega\|_{L^p(\Omega)}+\|\nabla p\|_{L^p(\Omega)}+\| |\HH||\nabla\HH|\|_{L^p(\Omega)} \nonumber\\
 &\qquad+ \|\nabla f\|_{L^p(\Omega)}+\|\Div\uu\|_{L^\infty(\Omega)}\|\nabla\rho\|_{L^p(\Omega)}). \nonumber
\end{align}

Now, on the one hand, we have that
\begin{align*}
&\|\nabla\F\|_{L^p(\Omega)}+\|\nabla\omega\|_{L^p(\Omega)}+\|\nabla p\|_{L^p(\Omega)}+\| |\HH||\nabla\HH|\|_{L^p(\Omega)} \\
&\qquad\qquad+ \|\Div\uu\|_{L^\infty(\Omega)}\|\nabla\rho\|_{L^p(\Omega)}\\
&\qquad\leq C(\|\nabla\dot{\uu}\|_{L^2(\Omega)}+\|\nabla^2\HH\|_{L^2(\Omega)}+1)(1+\|\nabla\rho\|_{L^p(\Omega)}).
\end{align*}

On the other hand, recalling identity \eqref{Jy/rho(x)}, we have
\begin{align}
\|\nabla f\|_{L^p(\Omega)}&\leq C(\|\nabla\rho\|_{L^p(\Omega)}+\|\E\cdot \nabla_\y\psi (t,\y(t,\x))\J_\y^{\frac{1}{p}}\|_{L^p(\Omega)}\nonumber\\
&\qquad\qquad\qquad+\|\E\cdot\nabla_\y\rho_0(\y(t,\x)))\J_\y^{\frac{1}{p}}\|_{L^p(\Omega)})\nonumber\\
&\leq C(\|\nabla\rho\|_{	L^p(\Omega)}+\|\E\|_{L^{2p}(\Omega)}(\|\nabla_\y\psi\|_{L^{2p}(\Omega_\y)}+\|\nabla_\y\rho_0\|_{L^{2p}(\Omega_\y)}))\nonumber\\
&\leq C(\|\nabla\rho\|_{L^p(\Omega)}+\|\E\|_{L^{2p}(\Omega)}(\|\psi\|_{H^{2}(\Omega_\y)}+\|\nabla\rho_0\|_{L^{2p}(\Omega)}))\nonumber\\
&\leq C(\|\nabla\rho\|_{L^p(\Omega)}+\|\E\|_{L^{2p}(\Omega)}).\label{nablafLp}
\end{align}

Therefore,
\begin{equation}
\|\nabla^2\uu\|_{L^p(\Omega)}\leq C(\|\nabla\dot{\uu}\|_{L^2(\Omega)}+\|\nabla^2\HH\|_{L^2(\Omega)}+1)(1+\|\nabla\rho\|_{L^p(\Omega)}+\|\E\|_{L^{2p}(\Omega)}).\label{nabla2uLp}
\end{equation}

In order to estimate the term $\|\nabla\uu\|_{L^\infty(\Omega)}$, we recall the following inequality by Beale, Kato and Majda (see \cite{BKM,HLX,Mei}, {\em cf.} \cite{KatoPonce}), asserting, in particular, for the  $2$-torus $\Omega$,  if $w\in L^2(\Omega)$ with 
$\nabla w\in W^{1,q}(\Omega)$ for some $q>2$, then
\[
\|\nabla w\|_{L^\infty(\Omega)}\leq C(\|\Div w\|_{L^\infty(\Omega)}+\|w\|_{L^\infty(\Omega)})\log (1+\|\nabla^2 w\|_{L^q(\Omega)}) + C\|\nabla w\|_{L^2(\Omega)} + C.
\]

Using the previous estimates we then obtain
\begin{align}
&\|\nabla\uu\|_{L^\infty(\Omega)}\nonumber\\
&\leq C(\|\Div\uu\|_{L^\infty(\Omega)}+\|\omega\|_{L^\infty(\Omega)})\log(e+\|\nabla^2\uu\|_{L^p(\Omega)})+C\|\nabla\uu\|_{L^2(\Omega)}+C\nonumber\\
&\leq C(\|\nabla\dot{\uu}\|_{L^2(\Omega)}+\|\nabla^2\HH\|_{L^2(\Omega)}+1)\log (e+\|\nabla\rho\|_{L^p(\Omega)}+\|\E\|_{L^{2p}(\Omega)})\nonumber\\
&\qquad\qquad+ C(\|\nabla\dot{\uu}\|_{L^2(\Omega)}+\|\nabla^2\HH\|_{L^2(\Omega)}+1).\label{nablauinfty}
\end{align}

Coming back to \eqref{nablarhopprelim} we get
\begin{align*}
&\frac{d}{dt}\|\nabla\rho\|_{L^p(\Omega)} \\
&\leq C\Big( (\|\nabla\dot{\uu}\|_{L^2(\Omega)}+\|\nabla^2\HH\|_{L^2(\Omega)}+1)\log (e+\|\nabla\rho\|_{L^p(\Omega)}+\|\E\|_{L^{2p}(\Omega)})\\
&\qquad\qquad+ C(\|\nabla\dot{\uu}\|_{L^2(\Omega)}+\|\nabla^2\HH\|_{L^2(\Omega)}+1)\Big)(1+\|\nabla\rho\|_{L^p(\Omega)}+\|\E\|_{L^{2p}(\Omega)}).
\end{align*}

At this point we recall inequality \eqref{ELpeq} satisfied by the deformation gradient, so that defining
\[
Z(t):=(1+\|\nabla\rho\|_{L^p(\Omega)}+\|\E\|_{L^{2p}(\Omega)}),
\]
we arrive at the inequality
\begin{equation}
\frac{d}{dt}Z(t)\leq C(\|\nabla\dot{\uu}\|_{L^2(\Omega)}+\|\nabla^2\HH\|_{L^2(\Omega)}+1)(1+Z(t))\log (1+Z(t)),
\end{equation}
which can be rewritten in terms of the function $G=\log(1+\log(1+Z(t)))$ as
\begin{equation}
\frac{d}{dt}G(t)\leq C(\|\nabla\dot{\uu}\|_{L^2(\Omega)}+\|\nabla^2\HH\|_{L^2(\Omega)}+1).
\end{equation}

Therefore, upon integrating in $t$ we arrive at
\begin{equation}
\sup_{0\leq t\leq T}(\|\nabla\rho\|_{L^p(\Omega)} + \|\E\|_{L^{2p}(\Omega)})\leq C,
\end{equation}
and from \eqref{nablauinfty} we obtain
\[
\int_0^T\|\nabla\uu(t)\|_{L^\infty(\Omega)}^2dt\leq C.
\]
\end{proof}

\begin{lemma}
There is a constant $C>0$ depending on $\|\rho_0\|_{H^2(\Omega)}$, $\|\uu_0\|_{H^2(\Omega)}$, $\|\HH_0\|_{H^2(\Omega)}$, $\|\psi_0\|_{H^2(\Omega_y)}$ and the uniform bounds form above and from below for the initial density, such that
\begin{equation}
\sup_{0\leq t\leq T}[\|(\rho(t),\uu(t),\HH(t))\|_{H^2(\Omega)}+\|\E(t)\|_{L^\infty(\Omega)}]\leq C.
\end{equation}
\end{lemma}

\begin{proof}
At this point, we may apply estimate \eqref{LinftyE} in order to conclude that
\begin{equation}
\|\E(t)\|_{L^\infty(\Omega)}\leq C.
\end{equation}

Next, by standard $L^2$ estimates on the elliptic equation \eqref{E2u}, and using \eqref{nablafLp}, we have
\begin{align}
\|\uu\|_{H^2(\Omega)}&\leq C(\|\uu\|_{L^2(\Omega)}+\|\rho\dot{\uu}\|_{L^2(\Omega)}+\|\nabla p\|_{L^2(\Omega)}\nonumber\\
&\qquad\qquad\qquad\qquad+\| |\HH| |\nabla\HH|\|_{L^2(\Omega)}+\|\nabla f\|_{L^2(\Omega)})\nonumber\\
&\leq  C(1+\|\nabla\HH\|_{L^4(\Omega)}) \leq \frac{1}{4}\|\HH\|_{H^2(\Omega)}+C.\label{uH2}
\end{align}

Similarly, considering the magnetic field equation \eqref{E2H} as an elliptic equation we have
\begin{align}
\|\HH\|_{H^2(\Omega)}&\leq C(\|\HH\|_{L^2(\Omega)}+\|\HH_t\|_{L^2(\Omega)}+\| |\HH| |\nabla\uu|\|_{L^2(\Omega)}+\| |\uu| |\nabla\HH|\|_{L^2(\Omega)})\nonumber\\
  &\leq  \frac{1}{2}\|\uu\|_{H^2(\Omega)}+\frac{1}{4}\|\HH\|_{H^2(\Omega)} +C.\label{HH2}
\end{align}

Combining \eqref{uH2} and \eqref{HH2} we get
\begin{equation}
\|\uu\|_{H^2(\Omega)}+\|\HH\|_{H^2(\Omega)}\leq C.
\end{equation}

Next, similar to the proof of the case $m=2$ of Lemma~\ref{higheu}, applying the operator $\nabla^2$ to the continuity equation \eqref{E2rho}, multiplying the resulting equation by $\nabla^2\rho$ and integrating we have
\begin{align}
\frac{d}{dt}\|\nabla^2\rho\|_{L^2(\Omega)}^2\leq C[(\|\nabla\uu\|_{L^\infty(\Omega)}+1)\|\nabla^2\rho\|_{L^2(\Omega)}^2+\|\nabla^3\uu\|_{L^2(\Omega)}^2+1].\label{nabla2rhoprelim}
\end{align}

Now we see that
\begin{align*}
\|\nabla^3\uu\|_{L^2(\Omega)}&\leq \|\nabla^2\Div\uu\|_{L^2(\Omega)}+\|\nabla^2\omega\|_{L^2(\Omega)})\\
 &\leq C(\|\nabla^2\F\|_{L^2(\Omega)}+\|\nabla^2(|\HH|^2)\|_{L^2(\Omega)}+\|\nabla^2 p\|_{L^2(\Omega)}+\|\nabla^2 f\|_{L^2(\Omega)}\\
 &\qquad+\|\nabla\rho\cdot\nabla\Div\uu\|_{L^2(\Omega)}+\|\Div\uu\nabla^2\rho\|_{L^2(\Omega)}  +\|\nabla^2\omega\|_{L^2(\Omega)}).
\end{align*}

Here we note that
\begin{align*}
 &\|\nabla^2\F\|_{L^2(\Omega)}+\|\nabla^2(|\HH|^2)\|_{L^2(\Omega)}+\|\nabla^2 p\|_{L^2(\Omega)}\\
 &\qquad+\|\nabla\rho\cdot\nabla\Div\uu\|_{L^2(\Omega)}+\|\Div\uu\nabla^2\rho\|_{L^2(\Omega)}  +\|\nabla^2\omega\|_{L^2(\Omega)}\\
 &\leq C( \|\nabla(\rho\dot{\uu})\|_{L^2(\Omega)} + \|\nabla\rho\|_{L^4(\Omega)}\|\nabla^2\uu\|_{L^2(\Omega)}^{\frac{1}{2}}\|\nabla^3\uu\|_{L^2(\Omega)}^{\frac{1}{2}}\\
 &\qquad+(\|\Div\uu\|_{L^\infty(\Omega)}+1)\|\nabla^2\rho\|_{L^2(\Omega)}+1 )\\
 &\leq \varepsilon\|\nabla^3\uu\|_{L^2(\Omega)} + C_\varepsilon(\|\nabla\uu\|_{L^\infty(\Omega)}+1)\|\nabla^2\rho\|_{L^2(\Omega)} + \|\nabla\dot{\uu}\|_{L^2(\Omega)}+1),
 \end{align*}
for arbitrary $\varepsilon>0$.

Also, since \eqref{nabla2uLp} implies that $\nabla^2 u\in L^2(0,t;L^4(\Omega))$, we may apply Corollary~\ref{equivnorms} with $m=2$, in order to obtain
\begin{align*}
\|\nabla^2 f\|_{L^2(\Omega)}&\leq C\Big(\|\rho\|_{H^2(\Omega)}\|\psi\circ\Y\|_{H^2(\Omega)}\|\rho_0(\y(t,\cdot)\|_{H^2(\Omega)})\\
&\leq C\Big(\|\nabla^2\rho\|_{L^2(\Omega)}\|\psi\|_{H^2(\Omega_\y)}\|\rho_0\|_{H^2(\Omega)}\Big)\\
&\leq C(\|\nabla^2\rho\|_{L^2(\Omega)}+1)
\end{align*}

Putting this information into \eqref{nabla2rhoprelim} we obtain
\begin{align}
\frac{d}{dt}\|\nabla^2\rho\|_{L^2(\Omega)}^2\leq C\Big( (\|\nabla\uu\|_{L^\infty(\Omega)}+1)(\|\nabla^2\rho\|_{L^2(\Omega)}+1) + \|\nabla\dot{\uu}\|_{L^2(\Omega)}+1)\Big),
\end{align}
and Gronwall's inequality yields
\begin{equation}
\|\nabla^2\rho\|_{L^2(\Omega)}^2\leq C.
\end{equation}
\end{proof}
\section{High order estimates, and conclusion}\label{S5}

Moving on to the higher order estimates, we have the following.
\begin{lemma} 
	For $m\geq 3$, there exists a positive constant $C$ depending on $T$, $m_0$, $\Vert \rho_0 \Vert_{H^m}$, $\Vert \uu_0 \Vert_{H^m}$, $\|\HH_0\|_{H^m}$ and $\Vert \psi_0 \Vert_{H_\y^m})$, such that
	\begin{align*}
	\Vert \rho \Vert_{L^\infty(0,T;H^m)}, \Vert (\uu,\HH) \Vert_{L^2(0,T;H^{m+1})\cap H^1(0,T;H^{m-1})}, \Vert \psi \Vert_{L^\infty(0,T;H^m)} \leq C.
	\end{align*}
\end{lemma}

\begin{proof}
	Let us argue by induction, observing that the case with $m=2$ was already proven in above. First, using the continuity equation we rewrite the momentum equation \eqref{E2u} as
	\begin{equation*}
	\rho \uu_t + L_\rho \uu = \HH\cdot\nabla\HH-\frac{1}{2}\nabla|\HH|^2-\rho(\uu \cdot \nabla)\uu - \nabla p - \nabla (g'(1/\rho)h(|\psi\circ\Y|^2)).
	\end{equation*}
	Applying the tensorial operator $\nabla^{m-1}$ on the equation above, one arrives at
	\begin{flalign*}
	&\rho \nabla^{m-1} \uu_t + L_\rho \nabla^{m-1} \uu = \nabla^{m-1}\left(\HH\cdot\nabla\HH-\frac{1}{2}\nabla|\HH|^2\right)- \nabla^{m-2}(\nabla \rho \otimes \uu_t)&\\
	&\qquad \qquad\qquad\qquad + \nabla^{m-1} \big(\rho (\uu\cdot \nabla) \uu\big) - \nabla^m p + \alpha\nabla^m (g'(1/\rho)h(|\psi\circ\Y|^2)).
	\end{flalign*}
	Thus, defining
 \[
   Z(t):=\int_0^t \|\uu(s)\|_{H^{m+1}(\Omega)}^2 +\|\uu_t(s)\|_{H^{m-1}(\Omega)}^2ds,
 \]	
by Lemma~\ref{highparabol} and the low order estimates proven up to this moment, for $0 < t < T'$,
	\begin{align}
	&Z(t) \leq C \big(T, m_0, \Vert \rho_0 \Vert_{H^2}, \Vert \psi_0 \Vert_{H^2}, \Vert \uu_0 \Vert_{H^2} \big) \nonumber \\
	&\qquad \Big( 1 + \Vert \uu_0 \Vert_{H^m}^2 +\| |\HH|^2 \|_{L^2(0,t;H^m)}+ \Vert \nabla \rho \otimes \uu_t \Vert_{L^2(0,t;H^{m-2})}^2   \nonumber \\
	&\qquad\qquad\qquad\qquad + \Vert \rho (\uu \cdot \nabla) \uu \Vert_{L^2(0,t;H^{m-1})}^2 +   \Vert p \Vert_{L^2(0,t;H^m)}^2 \nonumber \\
	&\qquad\qquad\qquad\qquad\qquad\qquad+ \Vert \alpha g'(1/\rho)h(|\psi\circ\Y|^2) \Vert_{L^2(0,t;H^m)}^2 \Big). \label{final}
	\end{align}
	
At this point, in light of Lemmas~\ref{higheu}, \ref{lemmaH(u)}, \ref{existschr}, Corollary~\ref{equivnorms}, and the classic Leibniz rule for Sobolev functions, it is not difficult to  estimate each term in the right-hand side above. Let us just point out how to handle the most interesting -- and illustrating -- ones.

First, let us consider the pressure term $\Vert p \Vert_{L^2(0,t;H^m)}$. Once $p(\rho) = a \rho^\gamma$, a chain rule in the same spirit of \eqref{chainrule} gives
$$\Vert p(t) \Vert_{H^m} \leq C_m \big(\Vert \rho(t) \Vert_{H^{m-1} \cap L^\infty}, \operatorname{Min}_\x \rho(\x, t) \big)  \big( 1 + \Vert \rho (t) \Vert_{H^m}).$$
Thus, by \eqref{Hmrho} of Lemma \ref{higheu},
\begin{align}
 &\Vert p(t) \Vert_{H^m(\Omega)} \leq C(T,\|\uu\|_{L^2(0,T;H^{m}(\Omega))},\|\rho_0\|_{H^m(\Omega)}, M_0)\nonumber\\
	&\qquad\qquad\qquad\qquad\qquad\qquad\cdot(1+\|\uu\|_{L^2(0,t;H^{m+1})}) \nonumber \\
	&\qquad\qquad\>\>\>\>\>\>\leq C(T,\|\uu\|_{L^2(0,T;H^{m}(\Omega))},\|\rho_0\|_{H^m(\Omega)}, M_0, m_0) \cdot (1+Z(t)^{1/2}). \nonumber
\end{align}
Consequently,
$$\Vert p(t) \Vert_{L^2(0,t;H^m)}^2 \leq C  \Big( 1 + \int_0^t Z(s) ds \Big),$$
for some $C = C(T,\|\uu\|_{L^2(0,T;H^{m}(\Omega))},\|\rho_0\|_{H^m(\Omega)}, M_0)$.

The other quantity that poses a curious difficulty is $\Vert \nabla \rho \otimes \uu_t \Vert_{L^2(0,t;H^{m-2})}^2$ when $m=3$. Observe that, by Gagliardo--Nirenberg inequality, and Lemma \ref{higheu} -- especially \eqref{rhoW14} --,
\begin{align*}
  \Vert \nabla \rho(t) \otimes \uu_t(t) \Vert_{H^1} &\leq c \Vert \nabla \rho(t) \Vert_{L^4} \Vert \uu_t(t) \Vert_{W^{1,4}} + c\Vert \nabla^2 \rho(t) \Vert_{L^2} \Vert \uu_t(t) \Vert_{L^\infty} \\
                                                    &\leq \varepsilon \Vert \uu_t(t) \Vert_{H^2} + C_\varepsilon  (1 + \Vert \uu_t(t) \Vert_{L^2}),
\end{align*}
with $C_\varepsilon = C(\varepsilon, T, m_0, \Vert \uu_0 \Vert_{H^2}, \Vert \rho_0 \Vert_{H^2}, \Vert \HH_0 \Vert_{H^2}, \Vert \psi \Vert_{H^2})$. Thus, taking squares and integrating yield
$$\Vert \nabla \rho \otimes \uu_t \Vert_{L^2(0,t;H^{m-2})}^2 \leq C_\varepsilon + \varepsilon Z(t),$$
with the constant $C_\varepsilon$ depending on the very same parameters as the previous one.\footnote{All this inconvenience extraordinarily only happens with $m=3$. Indeed, when $m > 3$, one can make a much easier estimate $\Vert \nabla \rho(t) \otimes \uu_t(t) \Vert_{H^{m-2}} \leq c \Vert \nabla \rho(t) \Vert_{L^\infty} \Vert \uu_t(t) \Vert_{H^{m-2}} + c \Vert \nabla \rho (t) \Vert_{H^{m-2}} \Vert \uu_t \Vert_{L^\infty} \leq c \Vert \rho(t) \Vert_{H^3} \Vert \uu_t(t) \Vert_{H^{m-2}} + c\Vert \rho(t)\Vert_{H^{m-1}} \Vert \uu_t (t) \Vert_{H^2} $. In virtue of the induction hypothesis, these terms bring no problem.   }

In any case, the sum of all terms arising in \eqref{final} may be majorized in a similar fashion as
$$\frac{1}{2}Z(t) + C + C\int_0^t Z(s) ds,$$
for some constant $C$ depending only on $T$, $m_0$, $\Vert \rho_0 \Vert_{H^m}$, $\Vert \uu_0 \Vert_{H^m}$, $\|\HH_0\|_{H^m}$ and $\Vert \psi_0 \Vert_{H_\y^m})$. Thus, the conclusion follows from Gronwall's inequality.
\end{proof}

With this lemma we have the necessary a priori estimates, which combined with the local result from Proposition~\ref{localexist}, complete the proof of Theorem~\ref{principalthm}. Note that the local result holds with $m\geq 3$ and the case $m=2$ from Theorem~\ref{principalthm} follows from the a priori estimates with $m=2$ by a standard compactness argument.

\section{A continuous dependence inequality}\label{S6}

	As a corollary of Theorem~\ref{principalthm}, coming back to the proof of Lemma~\ref{contraction} yields the following relative energy estimate.

\begin{theorem} \label{dependcont}
	Let $m\geq 3$ and let $[\rho^{(j)}, \uu^{(j)},\HH^{(j)}, \psi^{(j)}]$ be solutions of \eqref{E2rho}--\eqref{E2psi} with initial data $[\rho_0^{(j)}, \uu_0^{(j)},\HH_0^{(j)}, \psi_0^{(j)}]$, $j=1,2$, lying in the appropriate $H^m-$space and with the initial densities being positive. Then there exists a constant $C$ such that for any $0 \leq t \leq T$
	\begin{align*}
	&\int_\Omega |\rho^{(1)} - \rho^{(2)}|^2 d\x+\int_\Omega \rho^{(1)} | \uu^{(1)} - \uu^{(2)} |^2 d\x + \int_0^t\int_\Omega| \nabla (\uu^{(1)} -  \uu^{(2)}) |^2 d\x ds \\
	&+ \int_\Omega  | \HH^{(1)} - \HH^{(2)} |^2 d\x + \int_0^t\int_\Omega| \nabla (\HH^{(1)} -  \HH^{(2)}) |^2 d\x dt' + \int_{\Omega_\y} |\psi^{(1)} - \psi^{(2)}|^2 d\y \\ 
	&\leq C \cdot \Bigg( \int_\Omega |\rho_0^{(1)} - \rho_0^{(2)}|^2 d\x+ \int_\Omega \rho_0^{(1)} | \uu_0^{(1)} - \uu_0^{(2)}|^2 d\x  + \int_\Omega |\HH_0^{(1)} - \HH_0^{(2)}|^2d\x\\
	&\qquad\qquad\qquad\qquad + \int_{\Omega_\y} |\psi_0^{(1)}- \psi_0^{(2)}|^2 d\y \Bigg)	
\end{align*}
	where $C$ depends only on $T$, the $H^m$--norm of initial data, $\min \rho^{(1)}_0(\x)$ and $\min \rho^{(2)}_0(\x)$.
\end{theorem}

\end{document}